\documentclass[12pt,reqno]{article}

\usepackage[margin=1in]{geometry}

\newcommand{\DateOfPub}[1]{}

\usepackage[UKenglish]{babel}

\usepackage{amsfonts,amsmath,amsthm,amssymb} 
\usepackage{mathrsfs} 
\usepackage{stmaryrd} 
\usepackage{mathtools}

\usepackage{lscape} 

\usepackage{natbib}
\usepackage{url}

\usepackage{tikz-cd}
\usepackage{ascmac}
\usepackage{soul}

\usepackage{comment} 
\usepackage{graphicx} 
\usepackage{setspace}
\usepackage{color}

\definecolor{Raspberry}{HTML}{E30B5C}

\usepackage{multirow}

\allowdisplaybreaks[1]

\newtheorem{theorem}{Theorem}
\newtheorem{proposition}[theorem]{Proposition}
\newtheorem{lemma}[theorem]{Lemma}
\newtheorem{corollary}[theorem]{Corollary}

\theoremstyle{remark}
\newtheorem{remark}{Remark}

\theoremstyle{definition}
\newtheorem{assumption}{Assumption}
\newtheorem{definition}{Definition}

\newcommand\EatDot[1]{}

\usepackage{bbm}
\usepackage[title]{appendix}

\usepackage[justification=centering]{caption}

\usepackage[multiple]{footmisc}

\usepackage{newtxmath,newtxtext}
\usepackage{hyperref}

\newcommand{\bX}{\boldsymbol{X}}

\newcommand{\bZ}{\boldsymbol{Z}}

\newcommand{\bu}{\boldsymbol{u}}
\newcommand{\bv}{\boldsymbol{v}}
\newcommand{\bw}{\boldsymbol{w}}
\newcommand{\bx}{\boldsymbol{x}}

\newcommand{\bz}{\boldsymbol{z}}

\newcommand{\bvarepsilon}{\boldsymbol{\varepsilon}}
\newcommand{\bzeta}{\boldsymbol{\zeta}}
\newcommand{\btheta}{\boldsymbol{\theta}}

\newcommand{\bfA}{\mathbf{A}}

\newcommand{\bfI}{\mathbf{I}}

\newcommand{\bfU}{\mathbf{U}}

\newcommand{\bLambda}{\mathbf{\Lambda}}

\newcommand{\bSigma}{\mathbf{\Sigma}}

\newcommand{\bbI}{\mathbb{I}}

\newcommand{\bbM}{\mathbb{M}}

\newcommand{\bbS}{\mathbb{S}}
\newcommand{\bbT}{\mathbb{T}}

\newcommand{\calC}{\mathcal{C}}

\newcommand{\calM}{\mathcal{M}}

\newcommand{\calO}{\mathcal{O}}

\newcommand{\calT}{\mathcal{T}}

\newcommand{\N}{\mathbb{N}}

\newcommand{\R}{\mathbb{R}}

\newcommand{\E}{\mathbb{E}}
\renewcommand{\Pr}{\mathbb{P}}

\newcommand{\iid}{\textrm{i.i.d.}}
\newcommand{\Bernoulli}{\mathrm{Ber}}
\newcommand{\Gaussian}{\mathcal{N}}
\newcommand{\BM}{\boldsymbol{W}}
\newcommand{\RE}{H}

\newcommand{\tr}{\mathrm{tr}}
\newcommand{\diag}{\mathrm{diag}}

\newcommand{\diff}{\mathrm{d}}

\renewcommand{\epsilon}{\varepsilon}
\newcommand{\sgn}{\mathrm{sgn}}
\newcommand{\zero}{\mathbf{0}}

\newcommand{\er}{\mathbf{r}}

\newcommand{\Risk}{\mathcal{R}}
\newcommand{\PS}{B[R]}
\newcommand{\PSULLN}{B^{p_{n}}[R]}
\newcommand{\samplespace}{\bbM}

\newcommand{\parspace}{\bbT}
\newcommand{\paralgebra}{\calT}

\makeatletter
\renewenvironment{abstract}{%
  \if@twocolumn
    \section*{\abstractname}%
  \else
    \small
    \quotation
    \par\noindent{\bfseries\abstractname.}
  \fi}
  {\if@twocolumn\else\endquotation\fi}
\newenvironment{keywords}{%
  \if@twocolumn
    \par\noindent{\bfseries Keywords:}
  \else
    \small
    \quotation
    \par\noindent{\bfseries Keywords:}
  \fi}
  {\if@twocolumn\else\endquotation\fi}
\newenvironment{msc}{%
  \if@twocolumn
    \par\noindent{\bfseries Keywords:}
  \else
    \small
    \quotation
    \par\noindent{\bfseries MSC2020 Subject Classification:}
  \fi}
  {\if@twocolumn\else\endquotation\fi}

\makeatother

\title{Dimension-free uniform concentration bound for logistic regression}
\author{Shogo Nakakita\thanks{Komaba Institute for Science, University of Tokyo. Email address: \texttt{nakakita@g.ecc.u-tokyo.ac.jp}}}
\date{}

\begin{document}
\maketitle

\begin{abstract}
    We provide a novel dimension-free uniform concentration bound for the empirical risk function of constrained logistic regression. 
    Our bound yields a milder sufficient condition for a uniform law of large numbers than conditions derived by the Rademacher complexity argument and McDiarmid's inequality.
    The derivation is based on the PAC-Bayes approach with second-order expansion and Rademacher-complexity-based bounds for the residual term of the expansion.
\end{abstract}


\begin{msc}
    Primary 62J12; secondary 62F12
\end{msc}

\section{Introduction}

The logistic regression model is a fundamental binary classification model in statistics and machine learning.
Given a sequence of $\R^{p}\times\{0,1\}$-valued independent and identically distributed (i.i.d.) random variables $\{(\bX_{i},Y_{i}); i=1,\ldots,n\}$, it supposes $Y_{i}$ as conditionally Bernoulli-distributed random variables such that for some $\btheta\in\R^{p}$, for all $i=1,\ldots,n$ and $\bx\in\R^{p}$,
\begin{equation}\label{eq:logisticreg}
    Y_{i}|\bX_{i}=\bx\sim \Bernoulli(\sigma(\langle \bx,\btheta\rangle)),
\end{equation}
where $\sigma(t)=1/(1+\exp(-t))$ with $t\in\R$ is the link function.
Each component of $\btheta$ explains the relationship between the corresponding component of $\bX_{i}$ and the conditional probability $\Pr(Y_{i}=1|\bX_{i})$.
The model is widely used for academic and industrial purposes as its interpretations are simple.

Our interest is the estimation of $\btheta$ with good prediction performance under high-dimensional settings.
To estimate $\btheta$, we frequently consider the minimization problem of the following empirical risk function (or the $(-1/n)$-scaled log-likelihood function):
\begin{equation}\label{eq:risk:empirical}
    \Risk_{n}(\btheta):=\frac{1}{n}\sum_{i=1}^{n}\left(-Y_{i}\log\sigma(\langle \bX_{i},\btheta\rangle)-(1-Y_{i})\log(1-\sigma(\langle \bX_{i},\btheta\rangle)\right).
\end{equation}
We expect that the minimization of $\Risk_{n}(\btheta)$ is a good approximation of the minimization of the unknown population risk function $\Risk(\btheta):=\E[\Risk_{n}(\btheta)]$.
If it holds true, the minimizers of $\Risk_{n}(\btheta)$ also achieve small population risk and thus good prediction performance.
Under the low-dimensional setting where $p$ is fixed in $n$ and $n\to\infty$, the classical argument for maximum likelihood estimation validates this idea.
However, this idea becomes difficult to justify in situations where $p$ is large relative to $n$.
For example, \citet{sur2019modern} point out that the minimizers of the empirical risk $\Risk_{n}(\btheta)$ can perform poorly as the estimators of the minimizers of the population risk $\Risk(\btheta)$ under high-dimensional settings.
In contrast, our study examines when the minimization of $\Risk_{n}(\btheta)$ is a good approximation of the minimization of $\Risk(\btheta)$ even with large $p$.

In particular, we study uniform concentration bounds and a uniform law of large numbers as their corollary for $\Risk_{n}(\btheta)$ around $\Risk(\btheta)$ on $\PS$, where $\PS=\{\btheta'\in\R^{p};\|\btheta'\|_{2}\le R\}$ with $R\ge0$ is the known bounded parameter space.
Let us consider the following ball-constrained minimization problem (ball-constrained logistic regression) instead of the unconstrained minimization on $\R^{p}$:
\begin{equation}\label{eq:constrainedlr}
    \text{minimize }\Risk_{n}\left(\btheta\right)\text{ subject to }\left\|\btheta\right\|_{2}\le R.
\end{equation}
This is a smooth convex optimization problem on a bounded convex set; it has solutions, which we can find efficiently.
Note that constraints on balls or spheres are not only mild but also typical in previous studies \citep{kuchelmeister2024finite,hsu2024sample}.
If we can conclude that $\Risk_{n}(\btheta)$ is uniformly close to $\Risk(\btheta)$ on $\PS$, then solving the minimization problem \eqref{eq:constrainedlr} (i.e., maximum likelihood estimation with the parameter space $\PS$) is a good approximation of the minimization of $\Risk(\btheta)$ on $\PS$.
If the following uniform law of large numbers holds, then we can support this idea asymptotically:
\begin{equation}\label{eq:ulln}
    \lim_{n\to\infty}\sup_{\btheta\in\PS}\left|\Risk_{n}(\btheta)-\Risk(\btheta)\right|=0\text{ almost surely.}
\end{equation}
The uniform law \eqref{eq:ulln} is one of the most fundamental arguments in large-sample theory.
It concludes that the minimization of the empirical risk $\Risk_{n}(\btheta)$ is asymptotically equivalent to the minimization of the population risk $\Risk(\btheta)$ \citep{van2000asymptotic}.
To derive sufficient conditions for the uniform law under high-dimensional settings, we analyse non-asymptotic uniform concentration bounds on $\sup_{\btheta\in\PS}|\Risk_{n}(\btheta)-\Risk(\btheta)|$.

A reasonable hypothesis is that $\Risk_{n}(\btheta)$ is uniformly close to $\Risk(\btheta)$ if an intrinsic dimension of the problem is sufficiently small in comparison to $n$.
A promising intrinsic dimension is the effective rank of $\bSigma:=\E[\bX_{1}\bX_{1}^{\top}]$ defined as $\er(\bSigma):=\tr(\bSigma)/\|\bSigma\|$ as long as $\|\bSigma\|>0$, where $\|\bSigma\|$ is the spectral norm of $\bSigma$.
The effective rank often appears in recent studies on possibly infinite-dimensional statistical models \citep{koltchinskii2017concentration,bartlett2020benign,zhivotovskiy2024dimension}.
The classical Rademacher complexity argument and McDiarmid's inequality \citep[e.g.,][]{bartlett2002rademacher,wainwright2019high,bach2024learning} indeed yield a dimension-free uniform concentration bound such that for all $\delta\in(0,1]$, with probability at least $1-\delta$,
\begin{equation}\label{eq:classical}
    \sup_{\btheta\in\PS}\left|\Risk_{n}(\btheta)-\Risk(\btheta)\right|\le 2\sqrt{\frac{R^{2}\|\bSigma\|\er(\bSigma)}{n}}+\sqrt{\frac{8\left(1+R^{2}K^{2}\|\bSigma\|\er(\bSigma)\right)\log\delta^{-1}}{n}}
\end{equation}
under a boundedness assumption $\|\bX_{i}\|^{2}\le K^{2}\tr(\bSigma)$ almost surely for some $K>0$.
Note that $p$ does not appear on the right-hand side; therefore, this bound is dimension-free.
We observe that the effective rank here explains an intrinsic dimension of the problem better than $p$ itself.

However, the bound \eqref{eq:classical} is not tight when we consider the uniform law of large numbers \eqref{eq:ulln}.
The bound \eqref{eq:classical} combined with the Borel--Cantelli lemma leads to the uniform law $\eqref{eq:ulln}$ provided that $\er(\bSigma)\log n/n\to0$ holds and $K$, $R$, and $\|\bSigma\|$ are fixed in $n$.
This logarithmic dependence on $n$ should be redundant and improvable.
For example, we can consider a case where $\er(\bSigma)\asymp n/\log n$.
Since $\er(\bSigma)/n\to0$, we expect that the uniform law holds; however, $\er(\bSigma)\log n/n\asymp 1$ and the bound \eqref{eq:classical} does not yield the uniform law.

Our study gives a dimension-free uniform concentration bound yielding a milder sufficient condition for the uniform law of large numbers.
We derive a bound such that for some explicit $c>0$ independent of $\bSigma$, $n$, and $p$, for any $\delta\in(0,1]$, with probability at least $1-\delta$,
\begin{equation}
    \sup_{\btheta\in\PS}\left|\Risk_{n}\left(\btheta\right)-\Risk\left(\btheta\right)\right|\le c\left(\sqrt{\frac{\|\bSigma\|\er(\bSigma)+(1+\|\bSigma\|)(1+\log\delta^{-1})}{n}}+\frac{\|\bSigma\|\sqrt{1+\log\delta^{-1}}}{n}\right).
\end{equation}
It is noteworthy that this bound gives a milder and more natural sufficient condition $\er(\bSigma)/n\to0$ for the uniform law of large numbers than the classical bound \eqref{eq:classical}.

The proof of our result is based on the PAC-Bayes approach with second-order expansion; e.g., see \citet{catoni2007pac} and \citet{alquier2024user} for introductions to the PAC-Bayes approach.
The PAC-Bayes approach is applicable to dimension-free analysis for linear problems such as covariance estimation \citep[e.g.,][]{zhivotovskiy2024dimension} and linear regression with least square losses, where functions of interests and their posterior means can coincide.
We extend such approaches to nonlinear functions of interest \eqref{eq:risk:empirical} by considering the second-order expansion of risk functions.
It generalizes previous studies for linear problems via the PAC-Bayes approach, as second-order terms are independent of parameters or even disappear under linearity.

\subsection{Motivations}
The motivations of our studies are not only to unveil the behaviour of maximum likelihood estimation in high-dimensional logistic regression, as explained above, but also to address anisotropic input vectors, that is, $\{\bX_{i};i=1,\ldots,n\}$ with $\E[\bX_{i}\bX_{i}^{\top}]\neq \bfI_{p}$.
Although non-sparse high-dimensional logistic regression with isotropic $\bX_{i}$ has been widely studied \citep{sur2019modern,salehi2019impact,kuchelmeister2024finite,hsu2024sample}, the case of anisotropic $\bX_{i}$ has been examined by only a limited number of studies \citep{candes2020phase,emami2020generalization,zhao2022asymptotic}.
Hence, our understanding of anisotropic cases is still lacking.
Anisotropic inputs have the following motivations: (i) they are realistic in applications, and (ii) we can understand the roles of $p$ and intrinsic dimensions separately.

First, anisotropic input vectors are widely observed in applications, and thus, considering such inputs enables us to understand the properties of high-dimensional logistic regression in practical situations.
We here list two typical causes of anisotropic input vectors: essentially low-dimensional input vectors and dependent input vectors.
Many high-dimensional data are essentially embedded into low-dimensional spaces or manifolds in recent data science \citep{zhivotovskiy2024dimension}.
Besides, the components of input vectors are correlated in practice, and thus, the assumption of the isotropy of inputs seldom holds.
Hence, analysis without isotropy is important to be consistent with practice.

In the second place, anisotropic settings allow us to deepen the understanding of the roles of $p$ and intrinsic dimensions.
Under the isotropic situation, $p$ is not only the dimension of input vectors but also a natural intrinsic dimension of the problem.
Therefore, how $p$ is important compared to the intrinsic one remains unclear.
Our analysis implies that $\er(\bSigma)$ is a significant intrinsic dimension of the problem \eqref{eq:constrainedlr} even with anisotropic inputs.
Moreover, $p$ itself is irrelevant when we give the uniform concentration bounds for \eqref{eq:constrainedlr} and the uniform law \eqref{eq:ulln}.

\subsection{Literature review}
\subsubsection{High-dimensional generalized linear models}
Statistical inference for generalized linear models (GLMs) with large $p$ has been extensively studied.
There are four typical settings on $\btheta$: moderately high-dimensional $\btheta$ with small $p/n$ \citep{liang2012maximum,kuchelmeister2024finite,hsu2024sample}; sparse and high-dimensional $\btheta$ \citep{van2008high,james2009generalized,levy2023generalization}; non-sparse and proportionally high-dimensional $\btheta$, i.e., $p/n\to \kappa\in(0,\infty)$ \citep{sur2019modern,sur2019likelihood,salehi2019impact,aubin2020generalization,emami2020generalization,zhao2022asymptotic,sawaya2023statistical}; and non-sparse and possibly infinite-dimensional $\btheta$ \citep{wu2023finite}.
Note that our study adopts the last setting.

\citet{liang2012maximum} discuss the asymptotic properties of the maximum likelihood estimation of logistic regression with $p=o(n)$.
\citet{kuchelmeister2024finite} examine the non-asymptotic rates of convergence in the maximum likelihood estimation of constrained logistic regression under large and small signal-to-noise settings with small $p/n$.
\citet{hsu2024sample} study the non-asymptotic sample complexities of constrained logistic regression with several inverse temperature settings and small $p/n$.
\citet{van2008high} and \citet{james2009generalized} study LASSO and Dantzig-selector estimators for GLMs with sparsity respectively.
\citet{levy2023generalization} analyse the multinominal logistic regression problem under sparsity and show its generalization error.
\citet{sur2019modern} and \citet{sur2019likelihood} investigate logistic regression under proportionally high-dimensional settings via the approximate message passing theory, show the biased behaviours of the ordinary maximum likelihood estimation, and propose procedures for debiased estimation and hypothesis testing.
\citet{zhao2022asymptotic} extend the results of \citet{sur2019modern} and \citet{sur2019likelihood} to general covariance settings under proportional high-dimensionality.
\citet{candes2020phase} analyse the existence of the maximum likelihood estimator of proportionally high-dimensional logistic regression and its phase transition.
\citet{arsov2019stability} study stability and generalization errors for some models including logistic regression.
\citet{salehi2019impact} investigate the impact of some regularization schemes on logistic regression under proportionally high-dimensional settings via the convex Gaussian min-max theorem.
\citet{aubin2020generalization} consider some regularized empirical risk minimization problems including regularized logistic regression with proportionally high-dimensional settings and its Bayes optimality.
\citet{emami2020generalization} analyse the asymptotic generalization errors of GLMs under proportional high-dimensional limits solved by a version of the approximate message passing algorithm.
\citet{sawaya2023statistical} discuss the estimation of GLMs with general link functions with monotonicity and their asymmetric inverses.
\citet{wu2023finite} analyse GLMs with non-sparse and possibly infinite-dimensional $\btheta$ and the ReLU activation function as its link function trained by several algorithms.

\subsubsection{Concentration of infinite-dimensional statistical models}
Dimension-free bounds for non-sparse and possibly infinite-dimensional statistical models are a topical problem; note that most of the studies are based on linearity.
\citet{catoni2017dimension} consider the PAC-Bayes approach for estimation problems with possibly infinite-dimensional settings and heavy-tailed distributions and give several dimension-free bounds.
\citet{bartlett2020benign} show that the excess risk of the minimum $\ell^{2}$-norm interpolator for a linear regression model with $p>n$ gets small under the moderate decay of the eigenvalues of the covariance matrix of $\bX_{i}$'s.
\citet{tsigler2023benign} further study ridge regression with $p>n$ in detail and give tighter bounds for the excess risk than \citet{bartlett2020benign}.
\citet{cheng2022dimension} also investigate ridge regression and develop a unified non-asymptotic analysis applicable to both underparameterized ($p<n$) and overparameterized ($p>n$) schemes.

\subsection{Notation}
For any vector $\bv\in\R^{d}$, $\|\bv\|$ denotes the $\ell^{2}$-norm of $\bv$.
For any vectors $\bu,\bv\in\R^{d}$, $\langle \bu,\bv\rangle$ represents the dot product of $\bu$ and $\bv$.
For any matrix $\bfA\in\R^{d_{1}}\otimes\R^{d_{2}}$, $\|\bfA\|$ denotes the spectral norm of $\bfA$ and $\bfA^{\top}$ denotes the transpose of $\bfA$.
For any square matrix $\bfA\in\R^{d}\otimes\R^{d}$, $\tr(\bfA)$ is the trace of $\bfA$.
For any positive semi-definite matrix $\bfA\in\R^{d}\otimes\R^{d}$, $\er(\bfA)$ is the effective rank of $\bfA$ defined as
\begin{equation}
    \er(\bfA)=\begin{cases}
        \frac{\tr(\bfA)}{\|\bfA\|} & \text{ if }\|\bfA\|>0,\\
        0 & \text{ otherwise.}
    \end{cases}
\end{equation}

Let $B^{d}[r]:=\{\btheta\in\R^{d};\|\btheta\|\le r\}$ with $r\ge0$; we also use the notation $B[r]=B^{d}[r]$ if $d$ is obvious by the context.
We define $\bbS^{d-1}:=\{\bu\in\R^{d};\|\bu\|=1\}$ with $d\in\N$.

For any function $f:\R^{d_{1}}\to\R^{d_{2}}$, $f\in\calC^{r}(\R^{d_{1}};\R^{d_{2}})$ means that $f$ is $r$-times continuously differentiable, $f\in\calC_{b}^{r}(\R^{d_{1}};\R^{d_{2}})$ means that $f\in\calC^{r}(\R^{d_{1}};\R^{d_{2}})$ and $f$ as well as its derivatives of order up to $r$ are bounded, and $f\in\calC_{p}^{r}(\R^{d_{1}};\R^{d_{2}})$ means that $f\in\calC^{r}(\R^{d_{1}};\R^{d_{2}})$ and $f$ as well as its derivatives of order up to $r$ are of at most polynomial growth.
For any $f\in\calC^{1}(\R^{d};\R)$, $\nabla f:=(\partial_{1}f,\cdots,\partial_{d}f)$ is the gradient of $f$.
For any $f\in\calC^{2}(\R^{d};\R)$, $\Delta f:=\sum_{i=1}^{d}\partial_{i}^{2}f$ denotes the Laplacian of $f$.

For functions $f,g:\N\to[0,\infty)$, $f(n)=\calO(g(n))$ implies that for some $M>0$ and $n_{0}\in\N$, for all $n\ge n_{0}$, $f(n)/g(n)\le M$, $f(n)=o(g(n))$ implies that $f(n)/g(n)\to 0$,  $f(n)=\omega(g(n))$ implies that $f(n)/g(n)\to \infty$, and $f(n)\asymp g(n)$ implies that both $f(n)=\calO(g(n))$ and $g(n)=\calO(f(n))$ hold.

For any real-valued random variable $X$, $\|X\|_{\psi_{2}}:=\inf\{t>0;\E[\exp(X^{2}/t^{2})]\le 2\}$, which is the $2$-Orlicz norm of $X$.
For any $\R^{d}$-valued random variables $\bX$, $\|\bX\|_{\psi_{2}}:=\sup_{\bu\in\bbS^{d-1}}\|\langle \bX,\bu\rangle\|_{\psi_{2}}$
For two probability measures $\rho,\mu$ defined on an identical measurable space such that $\rho$ is absolutely continuous with respect to $\mu$, $\RE(\rho\|\mu)$ denotes the Kullback--Leibler divergence of $\rho$ from $\mu$.
For two random variables $X,Y$ defined on an identical measurable space such that the law of $X$ is absolutely continuous with respect to that of $Y$, $\RE(X\|Y)$ denotes the Kullback--Leibler divergence of the law of $X$ from that of $Y$.

\subsection{Paper outline}
The paper is organized as follows: Section \ref{sec:main} shows the main theoretical studies, Section \ref{sec:discussion} gives detailed discussions of the main results, Section \ref{sec:experiments} displays the results of our numerical experiments, Section \ref{sec:proof} provides the proof of our main theorem, and the proofs of preliminary results are postponed to Appendix.

\section{Main results} \label{sec:main}
We begin with preparatory definition and assumption in Section \ref{sec:assumption} and give our main theorem, a dimension-free uniform concentration bound, in Section \ref{sec:theorem} and its corollary, a uniform law of large numbers, in Section \ref{sec:corollary}.
Section \ref{sec:sketch} presents an outline of the proof of the main theorem.

\subsection{Assumption}\label{sec:assumption}

We display the definition of the concentration property of random vectors and our assumption on $\{(\bX_{i},Y_{i})\}$ before presenting the main results of this study.

\begin{definition}
    For any $d\in\N$, an $\R^{d}$-valued random variable $\bzeta$ satisfies the \textit{concentration property with constant }$K>0$ if for every $1$-Lipschitz function $\varphi:\R^{d}\to\R$, it holds that $\E[|\varphi(\bzeta)|]<\infty$ and for all $u>0$,
        \begin{equation}
            \Pr\left(\left|\varphi(\bzeta)-\E[\varphi(\bzeta)]\right|\ge u\right)\le 2\exp\left(-u^{2}/K^{2}\right).
        \end{equation}
\end{definition}
The definition of the concentration property is adopted from \citet{adamczak2015note}; \citet{ledoux2001concentration} gives a comprehensive overview of this type of concentration.
A well-known sufficient condition for the concentration property is that the law of $\bZ$ satisfies a logarithmic Sobolev inequality \citep[see Section 5.4.2 of][]{bakry2014analysis}.
For example, the $d$-dimensional standard Gaussian random vector satisfies the concentration property with constant $\sqrt{2}$ for arbitrary $d\in\N$ \citep[see also Section 5.4 of][]{boucheron2013concentration}.

\begin{assumption}\label{assumption:1}
    $\{(\bX_{i},Y_{i});i=1,\ldots,n\}$ is a sequence of $\R^{p}\times\{0,1\}$-valued i.i.d.~random variables with $\bSigma:=\E[\bX_{i}\bX_{i}^{\top}]$. For each $i=1,\ldots,n$, $\bX_{i}$ has the representation $\bX_{i}=\bfU\bLambda^{1/2}\bZ_{i}$, where $\bfU$ is a $p\times p $ orthogonal matrix, $\bLambda$ is a $p\times p $ positive semi-definite diagonal matrix, and $\{\bZ_{i}\}$ is a sequence of $p$-dimensional i.i.d.~random vectors with the following properties:
    \begin{itemize}
        \item[(i)] $\bSigma=\bfU\bLambda \bfU^{\top}$.
        \item[(ii)] $\bZ_{i}$ is isotropic, that is, $\E[\bZ_{i}\bZ_{i}^{\top}]=\bfI_{p}$.
        \item[(iii)] The $\R^{np}$-valued random variable $\bZ:=(\bZ_{1},\cdots,\bZ_{n})$ satisfies the concentration property with constant $K>0$.
    \end{itemize}
\end{assumption}
Note that we do not suppose $Y_{i}|\bX_{i}=\bx_{i}\sim\Bernoulli(\sigma(\langle \bx_{i},\btheta^{\ast}\rangle))$ for some $\btheta^{\ast}\in\R^{p}$ and thus we consider a possibly misspecified logistic regression problem. 

\begin{remark}
    If Assumption \ref{assumption:1} holds and $\E[\bZ_{i}]=0$, then $\|\bZ_{i}\|_{\psi_{2}}\le \sqrt{3}K$, which can be checked as follows: for every $\bu\in\bbS^{p-1}$,
    \begin{equation*}
        \E\left[\exp\left(\langle\bZ_{i},\bu\rangle^{2}/(3K^{2})\right)\right]=1+\int_{1}^{\infty}\Pr\left(|\langle\bZ_{i},\bu\rangle|\ge \sqrt{3K^{2}\log t}\right)\diff t\le 1+2\int_{1}^{\infty}t^{-3}\diff t= 2.
    \end{equation*}
    Otherwise, $\|\bZ_{i}\|_{\psi_{2}}\le\|\bZ_{i}-\E[\bZ_{i}]\|_{\psi_{2}}+\|\E[\bZ_{i}]\|_{\psi_{2}}\le\sqrt{3}K+\sup_{\bu\in\bbS^{p-1}}|\E[\langle \bZ_{i},\bu\rangle]|\le\sqrt{3}K+\|\E[\bZ_{i}\bZ_{i}^{\top}]\|^{1/2}\le 1+\sqrt{3}K$.
\end{remark}

\subsection{Main theorem}\label{sec:theorem}
We show a dimension-free uniform concentration inequality for the empirical risk function of logistic regression, which is the main theorem of this study.
\begin{theorem}\label{thm:ucb}
    Suppose that Assumption \ref{assumption:1} holds.
    For arbitrary $\delta\in(0,1/6]$, $n\in\N$, and $R\ge 0$,
    with probability at least $1-6\delta$,
    \begin{align}\label{eq:main}
        \sup_{\btheta\in\PS}\left|\Risk_{n}(\btheta)-\Risk(\btheta)\right|
        &\le \sqrt{\frac{27(\log\delta^{-1}+(1+\sqrt{3}K)^{2}(\tr(\bSigma)/12+\|\bSigma\|R^{2}\log\delta^{-1}))(1+6R^{2})}{n}}\notag\\
        &\quad+2R\sqrt{\frac{\tr(\bSigma)}{n}}+\sqrt{\frac{78 K^{2}\left(256+R^{2}\tr\left(\bSigma\right)\right)}{n}}+\frac{9K^{2}R\|\bSigma\|\sqrt{\log\delta^{-1}}}{n}.
    \end{align}
\end{theorem}
The proof is given in Section \ref{sec:mainproof}; we give a sketch of the proof in Section \ref{sec:sketch}.
A detailed discussion is postponed to Section \ref{sec:discussion}.

The characteristics of our bounds are its dimension-free property and explicit representation of numerical constants.
The bound given by Theorem \ref{thm:ucb} is only dependent on $\delta,K,\tr(\bSigma),\|\bSigma\|$, and $R$, and thus dimension-free.
Hence, we can extend our result for $\R^{p}$ to general separable Hilbert spaces by the dimension-free property and the monotonicity of probability measures \citep[for example, see][]{giulini2018robust}.
Note that every numerical constant is explicit; it is sometimes not easy to obtain such explicit constants in dimension-free analysis \citep[e.g., ][]{koltchinskii2017concentration}.
Explicit representation enables us to give uncertainty quantification of the population risk at a solution of the constrained minimization problem \eqref{eq:constrainedlr} provided that $K$ is known and deterministic or high-probability upper bounds for $\|\bSigma\|$ and $\tr(\bSigma)$ are available.

\subsection{Uniform law of large numbers}\label{sec:corollary}
We consider a sufficient condition for the uniform law of large numbers \eqref{eq:ulln} via Theorem \ref{thm:ucb}.
To deal with $p$ and $\bSigma$ possibly dependent on $n$, we define a double array of random variables satisfying Assumption \ref{assumption:1} for each $n$.
\begin{assumption}\label{assumption:2}
    For each $n\in\N$, $p_{n}\in\N$ and $\{(\bX_{i,n},Y_{i,n});i=1,\ldots,n\}$ is a sequence of $\R^{p_{n}}\times\{0,1\}$-valued i.i.d.~random variables with $\bSigma_{n}:=\E[\bX_{1,n}\bX_{1,n}^{\top}]$.
    For some $K>0$, $\bX_{i,n}=\bfU_{n}\bLambda_{n}^{1/2}\bZ_{i,n}$ holds for all $i=1,\ldots,n$ and $n\in\N$, where $\bfU_{n}$ is a $p_{n}\times p_{n}$ orthogonal matrix, $\bLambda_{n}$ is a $p_{n}\times p_{n}$ positive semi-definite diagonal matrix, and $\{\bZ_{i,n};i=1,\ldots,n\}$ is a sequence of $p_{n}$-dimensional i.i.d.~random vectors with the following properties: 
    (i) $\bSigma_{n}=\bfU_{n}\bLambda_{n}\bfU_{n}^{\top}$; 
    (ii) $\E[\bZ_{i,n}\bZ_{i,n}^{\top}]=\bfI_{p_{n}}$; 
    (iii) $\bZ_{n}:=(\bZ_{1,n},\cdots,\bZ_{n,n})$ satisfies the concentration property with constant $K$.
\end{assumption}

We define the empirical risk function under Assumption \ref{assumption:2} using $\{(\bX_{i,n},Y_{i,n})\}$ instead of $\{(\bX_{i},Y_{i})\}$ for each $n$ as follows: for each $\btheta\in\R^{p_{n}}$,
\begin{equation}\label{eq:risk:empirical:n}
    \Risk_{n}(\btheta):=\frac{1}{n}\sum_{i=1}^{n}\left(-Y_{i,n}\log\sigma(\langle \bX_{i,n},\btheta\rangle)-(1-Y_{i,n})\log(1-\sigma(\langle \bX_{i,n},\btheta\rangle)\right).
\end{equation}
The following corollary yields a sufficient condition of the uniform law of large numbers; it holds by Theorem \ref{thm:ucb} with letting $\delta=1/n^{2}$ and the Borel--Cantelli Lemma.
\begin{corollary}[Uniform law of large numbers]\label{cor:ulln}
    Suppose that Assumption \ref{assumption:2} holds and $\sup_{n\in\N}\|\bSigma_{n}\|<\infty$.
    Fix an arbitrary $R\ge 0$.
    If $\lim_{n\to\infty}\er(\bSigma_{n})/n=0$, then
    \begin{equation}
        \lim_{n\to\infty}\sup_{\btheta\in\PSULLN}\left|\Risk_{n}(\btheta)-\E[\Risk_{n}(\btheta)]\right|=0\text{ almost surely.}
    \end{equation}
\end{corollary}
We can also know that the minimum empirical risk converges to the minimum population risk almost surely if $\er(\bSigma_{n})/n\to0$ by Corollary \ref{cor:ulln}.
This is milder than reasonable sufficient conditions derived by classical arguments; we discuss it in Section \ref{sec:discussion}.

\subsection{Idea of the proof}\label{sec:sketch}
We exhibit the idea of the proof of Theorem \ref{thm:ucb} briefly.
We first use the PAC-Bayes approach by adopting $\{\BM_{t}^{\btheta};t\ge0\}$ with $\btheta\in\PS$ , $p$-dimensional Brownian motions starting at $\btheta$, as the prior and posteriors and see that for any $t\in(0,1]$, with high probability, for all $\btheta\in\PS$,
\begin{align}
    \E_{\BM}\left[\Risk\left(\BM_{t}^{\btheta}\right)\right]-\E_{\BM}\left[\Risk_{n}\left(\BM_{t}^{\btheta}\right)\right]\le \text{(PAC-Bayes complexity bound for $\Risk_{n}$)}.
\end{align}
In the second place, we use It\^{o}'s formula and obtain that for any $t\in(0,1]$, with high probability, for all $\btheta\in\PS$,
\begin{align}
    \Risk\left(\btheta\right)-\Risk_{n}\left(\btheta\right)&\le \text{(PAC-Bayes complexity bound for $\Risk_{n}$)}\notag\\
    &\quad+\frac{1}{2}\int_{0}^{t}\left(\E_{\BM}\left[\Delta\Risk_{n}\left(\BM_{s}^{\btheta}\right)\right]-\E_{\BM}\left[\Delta\Risk\left(\BM_{s}^{\btheta}\right)\right]\right)\diff s.
\end{align}
Note that the Laplacian of $\Risk_{n}$ is no more dependent on $\{Y_{i}\}$ and we can reduce the problem of $\{(\bX_{i},Y_{i})\}$ to that of $\{\bX_{i}\}$ or equivalently $\{\bZ_{i}\}$, whose behaviours are tractable by Assumption \ref{assumption:1}.
We thirdly give a Bernstein-type inequality of the supremum of the second term on the right-hand side in $\btheta$ around its expectation in $\bX$ and derive a bound such that for any $t\in(0,1]$, with high probability, for all $\btheta\in\PS$,
\begin{align}
    \Risk\left(\btheta\right)-\Risk_{n}\left(\btheta\right)&\le \text{(PAC-Bayes complexity bound for $\Risk_{n}$)}\notag \\
    &\quad+\frac{1}{2}\E_{\bX}\left[\sup_{\btheta\in\PS}\int_{0}^{t}\left(\E_{\BM}\left[\Delta\Risk_{n}\left(\BM_{s}^{\btheta}\right)\right]-\E_{\BM}\left[\Delta\Risk\left(\BM_{s}^{\btheta}\right)\right]\right)\diff s\right]+\text{(residual)}.
\end{align}
In the last place, we evaluate a Rademacher-complexity-based bound for the second term on the right-hand side.
Since the sign on the left-hand side does not affect the discussion, we complete the proof by selecting appropriate $t\in(0,1]$.

\begin{remark}
    When we consider linear problems such as covariance estimation \citep[e.g.,][]{zhivotovskiy2024dimension} and linear regression with least square losses, Laplacians are independent of parameters or even equal to zero.
    We thus can regard our approach as a natural extension of the PAC-Bayes approach for linear problems to nonlinear ones.
\end{remark}

\begin{remark}
    It is noteworthy that PAC-Bayes bounds with second-order expansion themselves appear in previous studies \citep[e.g.,][]{wang2018identifying,tsuzuku2020normalized}, which typically employ Taylor's theorem.
    In our case, we can yield similar conclusions via Taylor's theorem instead of It\^{o}'s formula to expand the expectation with respect to Gaussian vectors; however, It\^{o}'s formula gives simpler proofs than Taylor's theorem.
    One of the reasons is that we repeatedly employ Stein's lemma and the Hermite polynomials in our proofs.
    It\^{o}'s formula leads to the remainder with the integrand of the tractable form $\sigma(1-\sigma)H_{0}$, where $H_{d}$ with $d\in\N_{0}$ is the $d$-th Hermite polynomial.
    On the other hand, Taylor's theorem gives the remainder with the integrand of the form $\sigma(1-\sigma)(H_{2}+H_{0})$ resulting in lengthy proofs.
\end{remark}

\section{Discussions}\label{sec:discussion}
We compare our result with the classical argument by Rademacher complexities and McDiarmid's inequality and its extension to unbounded cases from the viewpoint of the uniform law of large numbers.
Dealing with Gaussian $\bX_{i}$ as previous studies, one may consider an extension of the classical argument by the concentration of anisotropic subgaussian random vectors other than our approach; hence, we here give such an extension for comparison.
In this section, we suppose that Assumption \ref{assumption:2} holds.
Furthermore, assume that each $\bZ_{i,n}$ is centred and its components are independent.

A uniform concentration bound \eqref{eq:classical} is of order $\calO(\sqrt{\er(\bSigma_{n})\log \delta^{-1}/n})$, and an extension of the classical argument to subgaussian cases also gives the same order.
Theorem 6.3.2 of \citet{vershynin2018high} gives that there exists an absolute constant $a>0$ such that for any $t\ge 0$,
\begin{equation}
    \Pr\left(\max_{i=1,\ldots,n}\left|\|\bfU_{n}\bLambda_{n}^{1/2}\bZ_{i,n}\|-\tr(\bSigma_{n})^{1/2}\right|\ge a\sqrt{K^{4}\|\bSigma_{n}\|^{2}(t+\log n)}\right)\le 2\exp(-t).
\end{equation}
Since $|-y\log\sigma(\langle x,\btheta\rangle)-(1-y)\log(1-\sigma(\langle x,\btheta\rangle))|^{2}\le 2(1+R^{2}\|x\|^{2})$ for all $y\in\{0,1\}$, $x\in\R^{p_{n}}$, and $\btheta\in\PSULLN$, with probability at least $1-2\delta$, it holds that $\max_{i=1,\ldots,n}\sup_{\btheta\in\PSULLN}|-Y_{i}\log\sigma(\langle \bX_{i},\btheta\rangle)-(1-Y_{i})\log(1-\sigma(\langle \bX_{i},\btheta\rangle))|^{2}\le 2(1+2R^{2}(\tr(\bSigma_{n})+a^{2}K^{4}\|\bSigma_{n}\|^{2}(\log n+\log\delta^{-1})))$.
Therefore, the classical argument \citep[see Section 4.5 of][]{bach2024learning} gives that for arbitrary $\delta>0$, with probability $1-3\delta$,
\begin{align}\label{eq:extendedClassical}
    &\sup_{\btheta\in\PSULLN}\left|\Risk_{n}(\btheta)-\E[\Risk_{n}(\btheta)]\right|\notag\\
    &\le 2\sqrt{\frac{R^{2}\|\bSigma_{n}\|\er(\bSigma_{n})}{n}}+\sqrt{\frac{8(1+2R^{2}\|\bSigma_{n}\|(\er(\bSigma_{n})+a^{2}K^{4}\|\bSigma_{n}\|(\log n+\log\delta^{-1})))\log\delta^{-1}}{n}}.
\end{align}
The right-hand side above is of order $\calO(\sqrt{\er(\bSigma_{n})\log \delta^{-1}/n})$, which is the same as the estimate \eqref{eq:classical}.

Corollary \ref{cor:ulln} shows that a sufficient condition for the uniform law of large numbers \eqref{eq:ulln} is $\er(\bSigma_{n})/n\to0$, which is milder than ones based on the estimates \eqref{eq:classical} and \eqref{eq:extendedClassical}.
Both the estimates \eqref{eq:classical} and \eqref{eq:extendedClassical} combined with the Borel--Cantelli lemma and $\delta=1/n^{2}$ give a sufficient condition of the uniform law of large numbers $\er(\bSigma_{n})\log n/n\to0$, which is worse than $\er(\bSigma_{n})/n\to0$.

\section{Numerical experiments}\label{sec:experiments}
We show numerical experiments to examine how the minimizer of the problem \eqref{eq:constrainedlr} performs under large $p/n$ but small $\er(\bSigma)/n$.
In particular, we analyse the performance in prediction and sign recovery.

Let us show the settings of the numerical experiment.
We set $p=3000$ and $n=1000$, and thus $p/n=3$.
We examine two settings of $\bSigma$: $\bSigma=\bSigma_{p}^{\text{rec}}:=\diag\{1,1/2,\ldots,1/p\}$; and $\bSigma=\bfI_{p}$ as a baseline.
Note that the spectral norms of these $\bSigma$'s are identical ($\|\bSigma_{p}^{\text{rec}}\|=\|\bfI_{p}\|=1$); on the other hand, their effective ranks differ greatly ($\er(\bSigma_{p}^{\text{rec}})=\sum_{i=1}^{p}i^{-1}\approx8.5838$ and $\er(\bfI_{p})=p=3\times10^{3}$).
We generate $\bX_{i}\sim^{\iid}\Gaussian(\zero,\bSigma)$, $\btheta\sim\text{Unif}(\mathbb{S}^{p-1})$ (the uniform distribution on $\bbS^{p-1}$), and $Y_{i}\sim \Bernoulli(\sigma(\beta\langle \bX_{i},\btheta\rangle))$ for all $i=1,\ldots,n$, where the inverse temperature $\beta$ \citep{hsu2024sample} is set to be $10^{3}$.
Let us set $R=1$; it gives a convex relaxation of the minimization on $\bbS^{p-1}$ of \citet{hsu2024sample} (in fact, all the solutions in our numerical experiments lie on $\bbS^{p-1}$).
We also generate $\{(\bX_{i}^{\text{test}},Y_{i}^{\text{test}});i=1,\ldots,1000\}$, sequences of $1000$ independent copies of $(\bX_{1},Y_{1})$, as test data.
The total iteration number in the experiments is set to 100 for both the settings of $\bSigma$.

We evaluate $\hat{\btheta}$, the minimizer of the problem \eqref{eq:constrainedlr} (or equivalently the constrained maximum likelihood estimator), by the performance in prediction and sign recovery rather than risk functions.
We analyse these two qualitative problems since they make comparing $\hat{\btheta}$ under different $\bSigma$ straightforward.
We measure the performance of $\hat{\btheta}$ by prediction precision, that is, the ratios of training data $\{(\bX_{i},Y_{i})\}$ and test data $\{(\bX_{i}^{\text{test}},Y_{i}^{\text{test}})\}$ with $Y_{i}=\bbI(\langle \bX_{i},\hat{\btheta}\rangle\ge 0)$ and $Y_{i}^{\text{test}}=\bbI(\langle \bX_{i}^{\text{test}},\hat{\btheta}\rangle\ge 0)$.
This is more interpretable than the values of risk functions themselves.
We also evaluate $\hat{\btheta}$ by seeing how it can recover the sign of the true values of $\btheta$.
This is an important task to understand how each component of $\bX_{i}$ affects $Y_{i}$ and also easier to interpret than risk functions.

\subsection{Prediction performance}
Table \ref{tab:experiment:prediction} summarizes the prediction performance of $\hat{\btheta}$.
Comparing the cases with $\bSigma=\bSigma_{p}^{\text{rec}}$ and $\bSigma=\bfI_{p}$, we observe that $\hat{\btheta}$ with $\bSigma=\bSigma_{p}^{\text{rec}}$ performs well in the prediction of test data, and the gap of the performance between the training phase and test phase is quite small.
The experiment indicates that the empirical risk on $\PS$ \eqref{eq:constrainedlr} is a good approximation of the population risk minimization on $\PS$ given $\er(\bSigma)/n\ll1$.

Note that $\hat{\btheta}$ with $\bSigma=\bfI_{p}$ overfits training data but $\hat{\btheta}$ with $\bSigma=\bSigma_{p}^{\text{rec}}$ does not.
Whilst data are linearly separable with high probability \citep{cover1964geometrical,sur2019modern,candes2020phase} regardless of the setting of $\bSigma$, the empirical risk minimization with $\bSigma=\bfI_{p}$ seeks for $\btheta$ achieving linear separation; since the scale of each element of $\bX_{i}$ is uniform, it is a reasonable result.
On the other hand, although there exist linearly separating $\btheta$ with $\bSigma=\bSigma_{p}^{\text{rec}}$ with high probability, such $\btheta$ do not necessarily achieve small $\Risk_{n}(\btheta)$ since the scales of the elements of $\bX_{i}$ are not uniform.
In the following sign recovery experiment, we can observe that $\hat{\btheta}$ prioritizes learning the important elements of $\btheta$ over perfectly fitting training data.

\subsection{Sign recovery performance}
Table \ref{tab:experiment:sign} describes the sign recovery performance of $\hat{\btheta}$.
We show the sign recovery performance for the first 10/100/500 elements and all the elements.
Note that $\sum_{i=1}^{500}i^{-1}/\sum_{j=1}^{p}j^{-1}\approx0.79136$ and thus the variances of the first 500 elements of $\bX_{i}$ occupies almost 80\% of the sum of all the variances if $\bSigma=\bSigma_{p}^{\text{rec}}$.
We also examine the mean weighted by the variances of the components of $\bX_{i}$; this is motivated by the non-uniformity of the importance of $\btheta$.
We observe that $\hat{\btheta}$ with $\bSigma=\bSigma_{p}^{\text{rec}}$ performs well in the sign recovery tasks except for the signs recovery for all elements.
If $\bSigma=\bSigma_{p}^{\text{rec}}$, then it should be difficult to recover the signs of $\btheta$ corresponding to the elements of $\bX_{i}$ with small variances; hence, this result is reasonable.
More notably, $\hat{\btheta}$ successfully recovers the signs of the important elements of $\btheta$ when $\bSigma=\bSigma_{p}^{\text{rec}}$.

\begin{table}[ht]
    \centering
    \begin{tabular}{l||l|c}\hline
         Measurements & $\bSigma=\bSigma_{p}^{\text{rec}}$ & $\bSigma=\bfI_{p}$ \\\hline\hline
         Correct prediction in training & .74707 & \textbf{1.0000}\\
         Correct prediction in test & \textbf{.71256} & .64489\\\hline
         Mean absolute differences & \textbf{.03619} & .35511\\
         \hline
    \end{tabular}
    \caption{The predictive performance of the minimizer.
    The numbers are the means taken over all 100 experiments.}
    \label{tab:experiment:prediction}
\end{table}

\begin{table}[ht]
    \centering
    \begin{tabular}{l|l||l|l}\hline
         Mean sign recovery & Definition & $\bSigma=\bSigma_{p}^{\text{rec}}$ & $\bSigma=\bfI_{p}$ \\\hline\hline
         First 10 elements &$\sum_{i=1}^{10}\bbI(\sgn(\btheta_{i}^{\ast})=\sgn(\hat{\btheta}_{i}))/10$ & \textbf{.93000} & .63400\\
         First 100 elements &$\sum_{i=1}^{100}\bbI(\sgn(\btheta_{i}^{\ast})=\sgn(\hat{\btheta}_{i}))/100$ & \textbf{.81250} & .64630\\
         First 500 elements &$\sum_{i=1}^{500}\bbI(\sgn(\btheta_{i}^{\ast})=\sgn(\hat{\btheta}_{i}))/500$ & \textbf{.69890} & .64604\\
         All the elements &$\sum_{i=1}^{p}\bbI(\sgn(\btheta_{i}^{\ast})=\sgn(\hat{\btheta}_{i}))/p$& .59644 & \textbf{.64480}\\\hline
         Weighted by variances &$\sum_{i=1}^{p}\bSigma_{ii}\bbI(\sgn(\btheta_{i}^{\ast})=\sgn(\hat{\btheta}_{i}))/\sum_{j=1}^{p}\bSigma_{ii}$ & \textbf{.79131} & .64480\\\hline
    \end{tabular}
    \caption{The sign recovery of the minimizer. 
    The numbers are the means taken over all 100 experiments.
    $\btheta^{\ast}$ and $\hat{\btheta}$ represent the true value of $\btheta$ generating output variables and the minimizer of the problem \eqref{eq:constrainedlr}.
    The subscript $i$ of vectors denotes the $i$-th element of them, and the subscript $ii$ of matrices denotes the $(i,i)$-th element of them (or the $i$-th diagonal element).}
    \label{tab:experiment:sign}
\end{table}

\section{Proof of the main theorem} \label{sec:proof}
We first introduce three preliminary results in Section \ref{sec:preliminary} and show the proof of Theorem \ref{thm:ucb} using them in Section \ref{sec:mainproof}.
The proofs of the preliminary results are given in Appendix.
\subsection{Preliminary results}\label{sec:preliminary}
\subsubsection{PAC-Bayes bounds with second-order expansion}\label{sec:pacbm}
Let $(\samplespace,\calM)$ be a measurable space and $\{V_{i};i=1\ldots,n\}$ be a sequence of $(\samplespace,\calM)$-valued i.i.d.~random variables with probability measure $P^{V}$.
Fix $R\ge0$.
We let $\{\BM_{t}^{\btheta};t\ge0\}$ with $\btheta\in B^{p}[R]$ represent a $p$-dimensional Brownian motion starting at $\btheta$, that is, $\BM_{t}^{\btheta}=\btheta+\BM_{t}$,
where $\{\BM_{t};t\ge0\}$ is the $p$-dimensional standard Brownian motion starting at zero.
\begin{proposition}[PAC-Bayes bounds with second-order expansion]\label{prop:pacbm}
    Fix $t>0$.
    Let $f:\samplespace\times\R^{p}\ni(v,\bw)\mapsto f(v,\bw)\in\R$ be a measurable function satisfying the following:
    \begin{itemize}
        \item[(i)] For $P^{V}$-almost all $v\in\samplespace$, the following hold: $f(v,\cdot)\in\mathcal{C}^{2}(\R^{p};\R)$, and the gradient and Laplacian of $f(v,\cdot)$ equal some measurable functions $\nabla_{\bw}f:\samplespace\times\R^{p}\to\R^{p}$ and $\Delta_{\bw}f:\samplespace\times\R^{p}\to\R$ such that
        \begin{itemize}
            \item[(a)] $\nabla_{\bw} f(v,\cdot)$ is at most of polynomial growth.
            \item[(b)] For all $\btheta\in\PS$,
            \begin{equation*}
            \int_{0}^{t}\E_{V,\BM}\left[\left|\left.\Delta_{\bw}f(V_{i},\bw)\right|_{\bw=\BM_{s}^{\btheta}}\right|\right]\diff s<\infty.
        \end{equation*}
        \end{itemize}
        \item[(ii)] For all $\bw\in\R^{p}$, $|f(V_{i},\bw)|<\infty$ almost surely, and for all $\btheta\in\PS$, $\E_{V}[|f(V_{i},\btheta)|]<\infty$.
        \item[(iii)] 
        For some $\eta:\R^{p}\to[0,\infty)$, for all $\lambda\in\R$ and $\bw\in\R^{p}$,
        \begin{equation*}
            \E_{V}\left[\exp\left(\lambda\left(f(V_{1},\bw)-\E_{V}\left[f(V_{1},\bw)\right]\right)\right)\right]\le \exp\left(\frac{\lambda^{2}\eta^{2}(\bw)}{8}\right),
        \end{equation*}
        and $\sup_{\btheta\in\PS}\E[\eta^{2}(\BM_{t}^{\btheta})]<\infty$.
    \end{itemize}
    Then with probability at least $1-\delta$, for all $\btheta\in\PS$,
    \begin{align}
        \frac{1}{n}\sum_{i=1}^{n}\E_{V}\left[f\left(V_{i},\btheta\right)\right]\le &\,\frac{1}{n}\sum_{i=1}^{n}f\left(V_{i},\btheta\right)+\sqrt{\frac{\sup_{\btheta\in\PS}\E[\eta^{2}(\BM_{t}^{\btheta})](R^{2}/(2t)+\log\delta^{-1})}{2n}}\notag\\
        &+\frac{1}{2n}\sum_{i=1}^{n}\int_{0}^{t}\left(\E_{\BM}\left[\left.\Delta_{\bw}f\left(V_{i},\bw\right)\right|_{\bw=\BM_{s}^{\btheta}}\right]-\E_{V,\BM}\left[\left.\Delta_{\bw}f\left(V_{i},\bw\right)\right|_{\bw=\BM_{s}^{\btheta}}\right]\right)\diff s,
    \end{align}
    and with probability at least $1-\delta$, for all $\btheta\in\PS$,
    \begin{align}
        \frac{1}{n}\sum_{i=1}^{n}f\left(V_{i},\btheta\right)\le &\,\frac{1}{n}\sum_{i=1}^{n}\E_{V}\left[f\left(V_{i},\btheta\right)\right]+\sqrt{\frac{\sup_{\btheta\in\PS}\E[\eta^{2}(\BM_{t}^{\btheta})](R^{2}/(2t)+\log\delta^{-1})}{2n}}\notag\\
        &+\frac{1}{2n}\sum_{i=1}^{n}\int_{0}^{t}\left(\E_{V,\BM}\left[\left.\Delta_{\bw}f\left(V_{i},\bw\right)\right|_{\bw=\BM_{s}^{\btheta}}\right]-\E_{\BM}\left[\left.\Delta_{\bw}f\left(V_{i},\bw\right)\right|_{\bw=\BM_{s}^{\btheta}}\right]\right)\diff s.
    \end{align}
\end{proposition}

\subsubsection{Bernstein-type inequalities}\label{sec:bernstein}

Let $\bz=(\bz_{1},\ldots,\bz_{n})\in\R^{np}$ and $\tilde{\BM}_{s}^{\btheta}=\bLambda^{1/2}\bfU^{\top}(\btheta+\BM_{s})$.
We set
\begin{align}
    G_{t}^{\btheta}(\bz)&:=\frac{1}{2n}\sum_{i=1}^{n}\int_{0}^{t}\E_{\BM}\left[\sigma\left(1-\sigma\right)\left(\left\langle \tilde{\BM}_{s}^{\btheta},\bz_{i}\right\rangle\right)\right]\left\langle \bLambda\bz_{i},\bz_{i}\right\rangle\diff s\notag\\
    &\quad-\frac{1}{2n}\sum_{i=1}^{n}\int_{0}^{t}\E_{\BM,\bZ}\left[\sigma\left(1-\sigma\right)\left(\left\langle \tilde{\BM}_{s}^{\btheta},\bZ_{i}\right\rangle\right)\left\langle \bLambda\bZ_{i},\bZ_{i}\right\rangle\right]\diff s.
\end{align}

We give the following Bernstein-type inequalities.
\begin{proposition}\label{prop:bernstein}
    Suppose that Assumption \ref{assumption:1} holds.
    For all $u>0$ and $t>0$,
    \begin{align}
        \Pr&\left(\left|\sup_{\btheta\in\PS}G_{t}^{\btheta}(\bZ)-\E\left[\sup_{\btheta\in\PS}G_{t}^{\btheta}(\bZ)\right]\right|\ge u\right)\notag \\
        &\le 2\exp\left(-\frac{n}{6K^{2}}\min\left\{\frac{u^{2}}{155t\left(256+R^{2}\tr\left(\bSigma\right)\right)},\frac{u}{5\sqrt{t}R\|\bSigma\|}\right\}\right),
    \end{align}
    and
    \begin{align}
        \Pr&\left(\left|\sup_{\btheta\in\PS}(-1)G_{t}^{\btheta}(\bZ)-\E\left[\sup_{\btheta\in\PS}(-1)G_{t}^{\btheta}(\bZ)\right]\right|\ge u\right)\notag \\
        &\le 2\exp\left(-\frac{n}{6K^{2}}\min\left\{\frac{u^{2}}{155t\left(256+R^{2}\tr\left(\bSigma\right)\right)},\frac{u}{5\sqrt{t}R\|\bSigma\|}\right\}\right).
    \end{align}
\end{proposition}

\subsubsection{Expectation of the suprema of the Laplacian gaps}\label{sec:expsup}

\begin{proposition}\label{prop:expsup}
    Suppose that Assumption \ref{assumption:1} holds.
    For any $t\in(0,1]$,
    \begin{equation}
        \E_{\bZ}\left[\sup_{\btheta\in\PS}G_{t}^{\btheta}\left(\bZ\right)\right]\le 2R\sqrt{\frac{\tr(\bSigma)}{n}},\quad
        \E_{\bZ}\left[\sup_{\btheta\in\PS}(-1)G_{t}^{\btheta}\left(\bZ\right)\right]\le 2R\sqrt{\frac{\tr(\bSigma)}{n}}.
    \end{equation}
\end{proposition}

\subsection{Proof of Theorem \ref{thm:ucb}}\label{sec:mainproof}

\begin{proof}[Proof of Theorem \ref{thm:ucb}]
We discuss a uniform upper bound for $\mathcal{R}(\btheta)-\mathcal{R}_{n}(\btheta)$ and see that the same argument holds for $\mathcal{R}_{n}(\btheta)-\mathcal{R}(\btheta)$ immediately by flipping the sign.

(Step 1) Since $-y\log(\sigma(t))-(1-y)\log((1-\sigma(t)))\in[0,\log2+|t|]$ for all $y\in\{0,1\}$ and $t\in\R$, 
\begin{equation}
    f(V_{i},\bw):=-Y_{i}\log(\sigma(\langle \bX_{i},\bw\rangle))-(1-Y_{i})\log(1-\sigma(\langle \bX_{i},\bw\rangle))
\end{equation}
with $V_{i}=(\bX_{i},Y_{i})$ (or equivalently $V_{i}=(\bZ_{i},Y_{i})$) for each $\bw\in\R^{p}$ is a subgaussian random variable such that $\|f(V_{i},\bw)\|_{\psi_{2}}\le \log2+\|\bZ_{i}\|_{\psi_{2}}\|\bLambda^{1/2}\bfU^{\top}\bw\|$.
Proposition 2.5.2 of \citet{vershynin2018high} (the flow (iv) $\Rightarrow^{\times 1}$ (i) $\Rightarrow^{\times 2}$ (ii), where $x>0$ in $\Rightarrow^{\times x}$ denotes the absolute scales in the proof) gives $\|f(V_{i},\bw)-\E_{V}[f(V_{i},\bw)]\|_{\psi_{2}}\le 3\|f(V_{i},\bw)\|_{\psi_{2}}$.
Using $\|\bZ_{i}\|_{\psi_{2}}\le 1+\sqrt{3}K$ by Assumption \ref{assumption:1},
Proposition 2.5.2 of \citet{vershynin2018high} (the flow (iv) $\Rightarrow^{\times 1}$ (i) $\Rightarrow^{\times 2}$ (ii) $\Rightarrow^{\times 1/(2\sqrt{e})}$ (iii) $\Rightarrow^{\times 1}$ (v)) gives that for any $\lambda\in\R$ and $\bw\in\R^{p}$,
\begin{equation}\label{eq:proof:main:subgauss}
    \E_{V}\left[\exp\left(\lambda\left(f(V_{i},\bw)-\E_{V}[f(V_{i},\bw)]\right)\right)\right]\le \exp\left(\frac{9\lambda^{2}}{e}\left(\log2+\left(1+\sqrt{3}K\right)\left\|\bLambda^{1/2}\bfU^{\top}\bw\right\|\right)^{2}\right).
\end{equation}
Thus we choose $\eta:\R^{p}\to[0,\infty)$ such that
\begin{equation}
    \eta^{2}(\bw):=\frac{72}{e}\left(\log2+\left(1+\sqrt{3}K\right)\left\|\bLambda^{1/2}\bfU^{\top}\bw\right\|\right)^{2}.
\end{equation}
Since
\begin{equation*}
    \E_{\BM}\left[\left\|\bLambda^{1/2}\bfU^{\top}\BM_{t}^{\btheta}\right\|^{2}\right]=\E_{\BM}\left[\left\langle \bSigma \left(\BM_{t}+\btheta\right),\BM_{t}+\btheta\right\rangle\right]=t\tr(\bSigma)+\langle \bSigma\btheta,\btheta\rangle,
\end{equation*}
for all $t>0$ and $\btheta\in\PS$, 
\begin{equation}
    \E_{\BM}\left[\eta^{2}\left(\BM_{t}^{\btheta}\right)\right]\le \frac{144}{e}\left(1+\left(1+\sqrt{3}K\right)^{2}\left(t\tr(\bSigma)+\|\bSigma\|R^{2}\right)\right)
\end{equation}
by the fact that $(a+b)^{2}\le 2(a^{2}+b^{2})$ for all $a,b\in\R$.

(Step 2) We apply Proposition \ref{prop:pacbm} for $\mathcal{R}(\btheta)-\mathcal{R}_{n}(\btheta)$.
For all $t>0$, Conditions (i) and (ii) of Proposition \ref{prop:pacbm} are immediate since the gradient and Laplacian of $f$ are
\begin{align*}
    \nabla_{\bw}f(V_{i},\bw)&=-Y_{i}\bX_{i}+\sigma\left(\left\langle \bX_{i},\bw\right\rangle\right)\bX_{i}=-Y_{i}\bfU\bLambda^{1/2}\bZ_{i}+\sigma\left(\left\langle \bfU\bLambda^{1/2}\bZ_{i},\bw\right\rangle\right)\bfU\bLambda^{1/2}\bZ_{i},\\
    \Delta_{\bw}f(V_{i},\bw)&=\left(\sigma\left(1-\sigma\right)\right)\left(\langle\bX_{i},\bw\rangle\right)\left\|\bX_{i}\right\|^{2}=\left(\sigma\left(1-\sigma\right)\right)\left(\left\langle \bfU\bLambda^{1/2}\bZ_{i},\bw\right\rangle\right)\left\langle\bLambda\bZ_{i},\bZ_{i}\right\rangle.
\end{align*}
Condition (iii) of Proposition \ref{prop:pacbm} holds for all $t>0$ by the discussion above.
For all $t>0$, with probability at least $1-\delta$, for all $\btheta\in\PS$,
\begin{align}
    \mathcal{R}(\btheta)-\mathcal{R}_{n}(\btheta)&\le \sqrt{\frac{27(1+(1+\sqrt{3}K)^{2}(t\tr(\bSigma)+\|\bSigma\|R^{2}))(R^{2}/(2t)+\log(\delta^{-1}))}{n}}\notag\\
    &\quad+\frac{1}{2n}\sum_{i=1}^{n}\int_{0}^{t}\left(\E_{\BM}\left[\left.\Delta_{\bw}f\left(V_{i},\bw\right)\right|_{\bw=\BM_{s}^{\btheta}}\right]-\E_{V,\BM}\left[\left.\Delta_{\bw}f\left(V_{i},\bw\right)\right|_{\bw=\BM_{s}^{\btheta}}\right]\right)\diff s.
\end{align}

(Step 3)
By Propositions \ref{prop:bernstein} and \ref{prop:expsup}, for all $t\in(0,1]$, with probability at least $1-3\delta$, for all $\btheta\in\PS$,
\begin{align}\label{eq:proof:main:witht}
    \mathcal{R}(\btheta)-\mathcal{R}_{n}(\btheta)&\le \sqrt{\frac{27(1+(1+\sqrt{3}K)^{2}(t\tr(\bSigma)+\|\bSigma\|R^{2}))(R^{2}/(2t)+\log(\delta^{-1}))}{n}}\notag\\
    &\quad+2R\sqrt{\frac{\tr(\bSigma)}{n}}+\sqrt{\frac{930tK^{2}\left(256+R^{2}\tr\left(\bSigma\right)\right)\log\delta^{-1}}{n}}+\frac{30\sqrt{t}K^{2}R\|\bSigma\|\log\delta^{-1}}{n}.
\end{align}
Letting $t=1/(12\log \delta^{-1})$ (note $\log\delta^{-1}>1$ by $\delta\le 1/6$), we obtain that with probability at least $1-3\delta$, for all $\btheta\in\PS$,
\begin{align}
    \mathcal{R}(\btheta)-\mathcal{R}_{n}(\btheta)
    &\le \sqrt{\frac{27(\log\delta^{-1}+(1+\sqrt{3}K)^{2}(\tr(\bSigma)/12+\|\bSigma\|R^{2}\log\delta^{-1}))(1+6R^{2})}{n}}\notag\\
    &\quad+2R\sqrt{\frac{\tr(\bSigma)}{n}}+\sqrt{\frac{78 K^{2}\left(256+R^{2}\tr\left(\bSigma\right)\right)}{n}}+\frac{9K^{2}R\|\bSigma\|\sqrt{\log\delta^{-1}}}{n}.
\end{align}

(Step 4) The same argument holds for $\Risk_{n}(\btheta)-\Risk(\btheta)$ by applying Propositions \ref{prop:pacbm}--\ref{prop:expsup}.

Hence, we obtain the conclusion.
\end{proof}

\begin{remark}
    We select $t=1/(12\log\delta^{-1})$, which approximately solves the following quadratic equation on $(0,\infty)$ under $K=\sqrt{2}$ (for standard Gaussian distributions), $R\gg1$, $R=\calO(1)$, $\|\bSigma\|=\calO(1)$, and $\tr(\bSigma)=\omega(1)$:
    \begin{equation*}
        1860t^{2}\left(256+R^{2}\tr\left(\bSigma\right)\right)\log\delta^{-1}=27(1+(1+\sqrt{6})^{2}(t\tr(\bSigma)+\|\bSigma\|R^{2}))(R^{2}/2+t\log(\delta^{-1})).
    \end{equation*}
    This equation lets the first and the third terms on the right-hand side of \eqref{eq:proof:main:witht} be equal to each other.
    Under different asymptotic settings, the optimal selection of $t$ can be different to our selection.
\end{remark}

\section*{Acknowledgements}
I gratefully acknowledge Pierre Alquier for his enlightening comments.
This work was supported by JSPS KAKENHI Grant Number JP24K02904 and JST CREST Grant Numbers JPMJCR21D2 and JPMJCR2115.

\begin{appendices}
\section{On the Hermite polynomials}
We present some known results on Stein's lemma and the Hermite polynomials used in the proofs of the technical results for readers' convenience.
Let us introduce the Hermite polynomials of degrees only up to 3: for any $x\in\R$,
$H_{0}(x)=1$, $H_{1}(x)=x$, $H_{2}(x)=x^{2}-1$, and $H_{3}(x)=x^{3}-3x$.
\begin{lemma}\label{lem:hermite}
    Suppose that $\zeta\sim\Gaussian(0,1)$.
    The following hold:
    \begin{enumerate}
        \item[(i)] For any $f\in\calC_{p}^{1}(\R;\R)$ and $d=0,1,2$, $\E[f'(\zeta)H_{d}(\zeta)]=\E[f(\zeta)H_{d+1}(\zeta)]$.
        \item[(ii)] For any $f\in\calC_{p}^{2}(\R;\R)$ and $d=0,1$, $\E[f''(\zeta)H_{d}(\zeta)]=\E[f(\zeta)H_{d+2}(\zeta)]$.
    \end{enumerate}
\end{lemma}

\begin{proof}
    The statement (i) with $d=0$ is nothing more than Stein's lemma.
    Since any $f,g\in\calC_{p}^{1}(\R;\R)$,
    \begin{equation*}
        \E[f'(\zeta)g(\zeta)]=\E[(fg)'(\zeta)-f(\zeta)g'(\zeta)]=\E[f(\zeta)(g(\zeta)\zeta-g'(\zeta))],
    \end{equation*}
    we can easily check that the other statements hold.
\end{proof}

\section{Proofs of the technical results}

\subsection{Proof of Proposition \ref{prop:pacbm}}

We give an extension of Theorem 4.1 of \citet{alquier2016properties} to cases with pointwise Hoeffding-type exponential moment bounds.
Let $(\parspace,\paralgebra)$ denote a measurable space representing a general parameter space.

\begin{lemma}[PAC-Bayes bound with a pointwise Hoeffding-type condition]\label{lem:arc16}
    Fix a probability measure $\mu$ on $(\parspace,\paralgebra)$ and a measurable function $\phi:\samplespace\times\parspace\to\R$.
    Assume that (i) for some measurable function $\eta:\parspace\to[0,\infty)$, for all $\lambda\in\R$ and $w\in\parspace$, $\int_{\samplespace}\exp(\lambda \phi(v,w))P^{V}(\diff v)\le \exp(\lambda^{2}\eta^{2}(w)/8)$ and (ii) $|\phi(v,w)|<\infty$ $P^{V}\otimes\mu$-almost everywhere.
    For any $\delta\in(0,1]$ and $\lambda>0$, with probability at least $1-\delta$, for any probability measure $\rho$ on $(\parspace,\paralgebra)$,
    \begin{equation}
        \sum_{i=1}^{n}\int_{\parspace}\phi\left(V_{i},w\right)\rho\left(\diff w\right)\le \frac{n\lambda}{8}\int_{\parspace}\eta^{2}(w)\rho\left(\diff w\right)+\frac{\RE\left(\rho\|\mu\right)+\log\delta^{-1}}{\lambda}.
    \end{equation}
\end{lemma}

\begin{proof}
    We obtain the following fact via a variational representation \citep[see Section 5.2 of][]{catoni2004statistical}: for any measurable function $h:\parspace\to\R$,
    \begin{equation}
        \sup_{b\in\R}\log\left(\int_{\parspace}\exp\left(\min\left\{b,h\left(w\right)\right\}\right)\mu\left(\diff w\right)\right)=\sup_{\rho}\left(\sup_{b\in\R}\int_{\parspace}\min\left\{b,h\left(w\right)\right\}\rho\left(\diff w\right)-\RE\left(\rho\|\mu\right)\right),
    \end{equation}
    where the supremum in $\rho$ is taken over all the probability measures on $(\parspace,\paralgebra)$ and both the left-hand side and the right-hand side can be infinite (we let $\infty-\infty=-\infty$ as \citealp{catoni2004statistical}).
    The variational representation, Fubini's theorem, and Condition (i) give that for all $\lambda\in\R$,
    \begin{align*}
        &\E_{V}\exp\sup_{\rho}\left(\E_{W\sim\rho}\left[\lambda\sum_{i=1}^{n}\phi\left(V_{i},W\right)- \frac{n\lambda^{2}\eta^{2}(W)}{8}\right]-\RE\left(\rho\|\mu\right)\right)\\
        &=\E_{V}\exp\sup_{\rho}\left(\sup_{b\in\R}\E_{W\sim\rho}\min\left\{b,\lambda\sum_{i=1}^{n}\phi\left(V_{i},W\right)- \frac{n\lambda^{2}\eta^{2}(W)}{8}\right\}-\RE\left(\rho\|\mu\right)\right)\\
        &=\E_{V}\exp\sup_{b\in\R}\log\E_{W\sim\mu}\exp\min\left\{b,\lambda\sum_{i=1}^{n}\phi\left(V_{i},W\right)- \frac{n\lambda^{2}\eta^{2}(W)}{8}\right\}\\
        &=\E_{V}\sup_{b\in\R}\E_{W\sim\mu}\exp\min\left\{b,\lambda \sum_{i=1}^{n}\phi\left(V_{i},W\right)- \frac{n\lambda^{2}\eta^{2}(W)}{8}\right\}\\
        &=\E_{V}\E_{W\sim\mu}\exp\left(\lambda\sum_{i=1}^{n}\phi\left(V_{i},W\right)- \frac{n\lambda^{2}\eta^{2}(W)}{8}\right)\\
        &=\E_{W\sim\mu}\E_{V}\exp\left(\lambda\sum_{i=1}^{n}\phi\left(V_{i},W\right)- \frac{n\lambda^{2}\eta^{2}(W)}{8}\right)\\
        &\le \E_{W\sim\mu}\exp\left(\frac{n\lambda^{2}\eta^{2}(W)}{8}- \frac{n\lambda^{2}\eta^{2}(W)}{8}\right)\\
        &=1.
    \end{align*}
    Therefore, by setting a random variable $\xi$ such that
    \begin{equation*}
        \xi:=\sup_{\rho}\left(\E_{W\sim\rho}\left[\lambda\sum_{i=1}^{n}\phi\left(V_{i},W\right)- \frac{n\lambda^{2}\eta^{2}(W)}{8}\right]-\RE\left(\rho\|\mu\right)\right),
    \end{equation*}
    Markov's inequality yields that $\Pr(\xi\ge \log\delta^{-1})\le \delta\E \exp(\xi)\le \delta$; this is the desired conclusion.
\end{proof}

\begin{proof}[Proof of Proposition \ref{prop:pacbm}]
    We give only the first statement since the second one is quite parallel by flipping the sign of $f$.
    We also consider only the case where $\sup_{\btheta\in\PS}\E[\eta^{2}(\BM_{t}^{\btheta})]>0$; otherwise the proof is trivial.
    It\^{o}'s formula gives that for all $\btheta\in\PS$,
    \begin{equation}\label{eq:lem:bm:ito}
        \E_{\BM}\left[f\left(V_{i},\BM_{t}^{\btheta}\right)\right]=f(V_{i},\btheta)+\frac{1}{2}\int_{0}^{t}\E_{\BM}\left[\left.\Delta_{\bw}f\left(V_{i},\bw\right)\right|_{\bw=\BM_{t}^{\btheta}}\right]\diff t
    \end{equation}
    almost surely since the local martingale term is martingale by (i)-(a), and the interchange of expectations holds by (i)-(b).
    We have
    \begin{equation}
        \RE\left(\BM_{t}^{\btheta}\left\|\BM_{t}^{\zero}\right.\right)=\frac{1}{2t}\left\|\btheta\right\|^{2}\le \frac{R^{2}}{2t}
    \end{equation}
    (we use the law of $\BM_{t}^{\zero}=\BM_{t}$ as the ``prior'' and those of $\BM_{t}^{\btheta}$ with $\btheta\in\PS$ as the ``posteriors'').
    By Lemma \ref{lem:arc16} (Condition (ii) of Lemma \ref{lem:arc16} is satisfied by Condition (ii) of the proposition), for any $t>0$ and $\lambda>0$, with probability at least $1-\delta$, for all $\btheta\in\PS$,
    \begin{align}
        \sum_{i=1}^{n}\left(\E_{V,\BM}\left[f\left(V_{i},\BM_{t}^{\btheta}\right)\right]-\E_{\BM}\left[f\left(V_{i},\BM_{t}^{\btheta}\right)\right]\right)
        &\le \frac{n\lambda\sup_{\btheta\in\PS}\E[\eta^{2}(\BM_{t}^{\btheta})]}{8}+\frac{\frac{R^{2}}{2t}+\log\delta^{-1}}{\lambda}
    \end{align}
    Combining this inequality with Eq.~\eqref{eq:lem:bm:ito}, dividing by $n$, and fixing 
    \begin{equation*}
        \lambda=\sqrt{\frac{8(R^{2}/(2t)+\log\delta^{-1})}{n\sup_{\btheta\in\PS}\E[\eta^{2}(\BM_{t}^{\btheta})]}},
    \end{equation*}
    we obtain the conclusion.
\end{proof}

\subsection{Proof of Proposition \ref{prop:bernstein}}

We prove several preliminary lemmas for Proposition \ref{prop:bernstein}.

\begin{lemma}\label{lem:localLipschitz}
    Suppose that Assumption \ref{assumption:1} holds.
    It holds that for all $t>0$, $\btheta\in\PS$, and $\bz=(\bz_{1},\ldots,\bz_{n})\in\R^{np}$,
    \begin{equation}
        \left\|\nabla_{\bz}G_{t}^{\btheta}(\bz)\right\|
        \le \sqrt{\frac{2t\left\|\bSigma\right\|}{\pi n}}\left(4+\sqrt{\frac{R^{2}}{16n}\sum_{i=1}^{n}\left\langle \bLambda\bz_{i},\bz_{i}\right\rangle}\right).
    \end{equation}
\end{lemma}

\begin{proof}
In this proof, we frequently use the notation $\zeta\sim\Gaussian(0,1)$, $\mu_{i}=\langle \btheta,\bfU\bLambda^{1/2}\bz_{i}\rangle$, and $\lambda_{s,i}=\sqrt{s\langle \bLambda\bz_{i},\bz_{i}\rangle}$ for all $s\in[0,t]$,
and the following fact:
\begin{equation}
    \left[\begin{matrix}
        \left\langle\tilde{\BM}_{s}^{\btheta},\bz_{i}\right\rangle\\
        \left\langle\tilde{\BM}_{s}^{\btheta},\bu_{i}\right\rangle
    \end{matrix}\right]\sim\Gaussian\left(\left[\begin{matrix}
        \left\langle \btheta,\bfU\bLambda^{1/2}\bz_{i}\right\rangle\\
        \left\langle \btheta,\bfU\bLambda^{1/2}\bu_{i}\right\rangle
    \end{matrix}\right],
    s\left[\begin{matrix}
        \left\langle \bLambda\bz_{i},\bz_{i}\right\rangle & \left\langle \bLambda\bz_{i},\bu_{i}\right\rangle \\
        \left\langle \bLambda\bz_{i},\bu_{i}\right\rangle & \left\langle \bLambda\bu_{i},\bu_{i}\right\rangle
    \end{matrix}\right]\right).
\end{equation}
We choose arbitrary $\bu=(\bu_{1},\ldots,\bu_{n})\in\bbS^{np-1}$ and give a uniform bound for
\begin{align}
    \left\langle \nabla_{\bz}G_{t}^{\btheta}(\bz),\bu\right\rangle&=\frac{1}{n}\sum_{i=1}^{n}\left\langle \bLambda\bz_{i},\bu_{i}\right\rangle\int_{0}^{t}\E\left[\sigma\left(1-\sigma\right)\left(\left\langle \tilde{\BM}_{s}^{\btheta},\bz_{i}\right\rangle\right)\right]\diff s\notag\\
    &\quad+\sum_{i=1}^{n}\frac{1}{2n}\left\langle \bLambda\bz_{i},\bz_{i}\right\rangle \int_{0}^{t}\E\left[\left(\sigma\left(1-\sigma\right)\right)'\left(\left\langle \tilde{\BM}_{s}^{\btheta},\bz_{i}\right\rangle\right)\left\langle \tilde{\BM}_{s}^{\btheta},\bu_{i}\right\rangle\right]\diff s\notag\\
    &=\frac{1}{n}\sum_{i=1}^{n}\left\langle \bLambda\bz_{i},\bu_{i}\right\rangle\int_{0}^{t}\E\left[\sigma\left(1-\sigma\right)\left(\left\langle \tilde{\BM}_{s}^{\btheta},\bz_{i}\right\rangle\right)\right]\diff s\notag\\
    &\quad+\sum_{i=1}^{n}\frac{1}{2n}\left\langle \bLambda\bz_{i},\bz_{i}\right\rangle \left\langle \btheta,\bfU\bLambda^{1/2}\bu_{i}\right\rangle\int_{0}^{t}\E\left[\left(\sigma\left(1-\sigma\right)\right)'\left(\left\langle \tilde{\BM}_{s}^{\btheta},\bz_{i}\right\rangle\right)\right]\diff s\notag\\
    &\quad+\sum_{i=1}^{n}\frac{1}{2n}\left\langle \bLambda\bz_{i},\bu_{i}\right\rangle \int_{0}^{t}\E\left[\left(\sigma\left(1-\sigma\right)\right)'\left(\left\langle \tilde{\BM}_{s}^{\btheta},\bz_{i}\right\rangle\right)\left(\left\langle\tilde{\BM}_{s}^{\btheta},\bz_{i}\right\rangle-\left\langle \btheta,\bfU\bLambda^{1/2}\bz_{i}\right\rangle\right)\right]\diff s\notag\\
    &=\frac{1}{n}\sum_{i=1}^{n}\left\langle \bLambda\bz_{i},\bu_{i}\right\rangle\int_{0}^{t}\E\left[\sigma\left(1-\sigma\right)\left(\mu_{i}+\lambda_{s,i}\zeta\right)\right]\diff s\notag\\
    &\quad+\sum_{i=1}^{n}\frac{1}{2n}\left\langle \bLambda\bz_{i},\bz_{i}\right\rangle \left\langle \btheta,\bfU\bLambda^{1/2}\bu_{i}\right\rangle\int_{0}^{t}\E\left[\left(\sigma\left(1-\sigma\right)\right)'\left(\mu_{i}+\lambda_{s,i}\zeta\right)\right]\diff s\notag\\
    &\quad+\sum_{i=1}^{n}\frac{1}{2n}\left\langle \bLambda\bz_{i},\bu_{i}\right\rangle \int_{0}^{t}\E\left[\left(\sigma\left(1-\sigma\right)\right)'\left(\mu_{i}+\lambda_{s,i}\zeta\right)\lambda_{s,i}\zeta\right]\diff s\label{eq:proof:lem:localLipschitz:decomp};
\end{align}
the interchange of limiting operations in the first equality holds since $\|(\sigma(1-\sigma))'\|_{\infty}\|\tilde{\BM}_{s}^{\btheta}\|$ dominates the integrand, and the second equality holds since for any $s\in(0,t]$,
\begin{align*}
    &\E\left[\left(\sigma\left(1-\sigma\right)\right)'\left(\left\langle \tilde{\BM}_{s}^{\btheta},\bz_{i}\right\rangle\right)\left\langle \tilde{\BM}_{s}^{\btheta},\bu_{i}\right\rangle\right]\\
    &=\E\left[\left(\sigma\left(1-\sigma\right)\right)'\left(\left\langle \tilde{\BM}_{s}^{\btheta},\bz_{i}\right\rangle\right)\left(\left\langle \btheta,\bfU\bLambda^{1/2}\bu_{i}\right\rangle+\frac{\left\langle \bLambda\bz_{i},\bu_{i}\right\rangle}{\left\langle \bLambda\bz_{i},\bz_{i}\right\rangle}\left(\left\langle\tilde{\BM}_{s}^{\btheta},\bz_{i}\right\rangle-\left\langle \btheta,\bfU\bLambda^{1/2}\bz_{i}\right\rangle\right)\right)\right].
\end{align*}
We use the following fact repetitively: for any $\{a_{1},\ldots,a_{n}\}\subset\R^{n}$, $\{b_{1},\ldots,b_{n}\}\subset[0,\infty)^{n}$ with $|a_{i}|\le b_{i}$, $\bv=(\bv_{1},\ldots,\bv_{n})\in\R^{np}$ and $\bu=(\bu_{1},\ldots,\bu_{n})\in\bbS^{np-1}$,
by choosing $\bar{\bu}=(\bar{\bu}_{1},\ldots,\bar{\bu}_{n})=(\sgn(\langle \bv_{1},\bu_{1}\rangle)\bu_{1},\ldots,\sgn(\langle \bv_{n},\bu_{n}\rangle)\bu_{n})\in\bbS^{np-1}$,
\begin{align}\label{eq:proof:lem:localLipschitz:variational}
    \sum_{i=1}^{n}a_{i}\langle \bu_{i},\bv_{i}\rangle \le \sum_{i=1}^{n}|a_{i}\langle \bu_{i},\bv_{i}\rangle|\le\sum_{i=1}^{n}b_{i}\langle \bar{\bu}_{i},\bv_{i}\rangle \le \sup_{\tilde{\bu}\in\bbS^{np-1}}\sum_{i=1}^{n}b_{i}\langle \tilde{\bu}_{i},\bv_{i}\rangle=\sqrt{\sum_{i=1}^{n}b_{i}^{2}\|\bv_{i}\|^{2}}.
\end{align}

(Step 1) We give a bound for the first term on the most right-hand side of \eqref{eq:proof:lem:localLipschitz:decomp}.
Stein's lemma. the fact $\sigma'=\sigma(1-\sigma)$, H\"{o}lder's inequality, $\|\sigma\|_{\infty}\le 1$, and $\E[|\zeta|]=\sqrt{2/\pi}$ give that
\begin{align*}
    \left|\E\left[\sigma\left(1-\sigma\right)\left(\mu_{i}+\lambda_{s,i}\zeta\right)\right]\right|&=\left|\E\left[\sigma'\left(\mu_{i}+\lambda_{s,i}\zeta\right)\right]\right|=\frac{1}{\lambda_{s,i}}\left|\E\left[\partial_{\zeta}\sigma\left(\mu_{i}+\lambda_{s,i}\zeta\right)\right]\right|\\
    &=\frac{1}{\lambda_{s,i}}\left|\E\left[\sigma\left(\mu_{i}+\lambda_{s,i}\zeta\right)\zeta\right]\right|\le \frac{1}{\lambda_{s,i}}\sqrt{\frac{2}{\pi}}.
\end{align*}
Therefore, for some $\bu'=(\bu_{1}',\ldots,\bu_{n}')\in\bbS^{np-1}$,
\begin{align*}
    \frac{1}{n}\sum_{i=1}^{n}\left\langle \bLambda\bz_{i},\bu_{i}\right\rangle\int_{0}^{t}\E\left[\sigma\left(1-\sigma\right)\left(\mu_{i}+\lambda_{s,i}\zeta\right)\right]\diff s
    &= \frac{1}{n}\sum_{i=1}^{n}\left\langle \bLambda\bz_{i},\bu_{i}\right\rangle\int_{0}^{t}\frac{1}{\lambda_{s,i}}\E\left[\sigma\left(\mu_{i}+\lambda_{s,i}\zeta\right)\zeta\right]\diff s\\
    &\le \frac{\sqrt{2/\pi}}{n}\sum_{i=1}^{n}\diff s\left\langle \bLambda\bz_{i},\bu_{i}'\right\rangle\int_{0}^{t}\frac{1}{\sqrt{s\langle \bLambda\bz_{i},\bz_{i}\rangle}}\diff s\\
    &\le \frac{2\sqrt{2/\pi}}{n}\sum_{i=1}^{n}\frac{\sqrt{t}}{\sqrt{\left\langle\bLambda\bz_{i},\bz_{i}\right\rangle}}\diff s\left\langle \bLambda\bz_{i},\bu_{i}'\right\rangle\\
    &\le 2\sqrt{\frac{2t}{\pi n}}\sqrt{\frac{1}{n}\sum_{i=1}^{n}\left\|\frac{1}{\sqrt{\left\langle\bLambda\bz_{i},\bz_{i}\right\rangle}}\bLambda\bz_{i}\right\|^{2}}\\
    &\le 2\sqrt{\frac{2t\left\|\bSigma\right\|}{\pi n}},
\end{align*}
where the second inequality is by the fact \eqref{eq:proof:lem:localLipschitz:variational}.

(Step 2) In the next place, we give an upper bound for the second term on the most right-hand side of \eqref{eq:proof:lem:localLipschitz:decomp}.
Stein's lemma yields that
\begin{align*}
    \left|\E\left[\left(\sigma\left(1-\sigma\right)\right)'\left(\mu_{i}+\lambda_{s,i}\zeta\right)\right]\right|
    &=\left|\frac{1}{\lambda_{s,i}}\E\left[\lambda_{s,i}\left(\sigma\left(1-\sigma\right)\right)'\left(\mu_{i}+\lambda_{s,i}\zeta\right)\right]\right|\\
    &=\frac{1}{\lambda_{s,i}}\left|\E\left[\partial_{\zeta}\left(\sigma\left(1-\sigma\right)\right)\left(\mu_{i}+\lambda_{s,i}\zeta\right)\right]\right|\\
    &=\frac{1}{\lambda_{s,i}}\left|\E\left[\left(\sigma\left(1-\sigma\right)\right)\left(\mu_{i}+\lambda_{s,i}\zeta\right)\zeta\right]\right|\\
    &\le \frac{\sqrt{2/\pi}}{4\lambda_{s,i}},
\end{align*}
where the last inequality uses H\"{o}lder's inequality and $\|\sigma(1-\sigma)\|_{\infty}\le 1/4$.
For some $\bu''=(\bu_{1}'',\ldots,\bu_{n}'')\in\bbS^{np-1}$, by the fact \eqref{eq:proof:lem:localLipschitz:variational},
\begin{align*}
    &\sum_{i=1}^{n}\frac{1}{2n}\left\langle \bLambda\bz_{i},\bz_{i}\right\rangle \left\langle \btheta,\bfU\bLambda^{1/2}\bu_{i}\right\rangle\int_{0}^{t}\E\left[\left(\sigma\left(1-\sigma\right)\right)'\left(\mu_{i}+\lambda_{s,i}\zeta\right)\right]\diff s\\
    &\le \sum_{i=1}^{n}\frac{1}{2n}\left\langle \bLambda\bz_{i},\bz_{i}\right\rangle \left\langle \btheta,\bfU\bLambda^{1/2}\bu_{i}''\right\rangle\int_{0}^{t}\frac{\sqrt{2/\pi}}{4\sqrt{s\left\langle\bLambda\bz_{i},\bz_{i}\right\rangle}}\diff s\\
    &\le \sum_{i=1}^{n}\frac{1}{2n}\left\langle \bLambda\bz_{i},\bz_{i}\right\rangle\frac{\sqrt{2t/\pi}}{2\sqrt{\left\langle\bLambda\bz_{i},\bz_{i}\right\rangle}}\left\langle \btheta,\bfU\bLambda^{1/2}\bu_{i}''\right\rangle\\
    &= \frac{\sqrt{2t/\pi}}{4n}\sum_{i=1}^{n}\left\langle\sqrt{\left\langle \bLambda\bz_{i},\bz_{i}\right\rangle} \bLambda^{1/2}\bfU^{\top}\btheta,\bu_{i}''\right\rangle\\
    &\le \frac{\sqrt{2t/\pi}}{4n}\left\|\left(\sqrt{\left\langle \bLambda\bz_{1},\bz_{1}\right\rangle}\bLambda^{1/2}\bfU^{\top}\btheta,\ldots,\sqrt{\left\langle \bLambda\bz_{n},\bz_{n}\right\rangle}\bLambda^{1/2}\bfU^{\top}\btheta\right)\right\|\\
    &=\frac{\sqrt{2t/\pi}}{4n}\sqrt{\sum_{i=1}^{n}\left\|\sqrt{\left\langle \bLambda\bz_{i},\bz_{i}\right\rangle}\bLambda^{1/2}\bfU^{\top}\btheta\right\|^{2}}\\
    &= \frac{\sqrt{2t/\pi}}{4n}\sqrt{\sum_{i=1}^{n}\left\langle \bLambda\bz_{i},\bz_{i}\right\rangle\left\|\bLambda^{1/2}\bfU^{\top}\btheta\right\|^{2}}\\
    &\le \sqrt{\frac{2t\left\|\bSigma\right\|}{\pi n}}\sqrt{\frac{R^{2}}{16n}\sum_{i=1}^{n}\left\langle \bLambda\bz_{i},\bz_{i}\right\rangle}.
\end{align*}

(Step 3) We examine the third term on the most right-hand side of \eqref{eq:proof:lem:localLipschitz:decomp}.
It holds that
\begin{align*}
    \left|\E\left[\left(\sigma\left(1-\sigma\right)\right)'\left(\mu_{i}+\lambda_{s,i}\zeta\right)\lambda_{s,i}\zeta\right]\right|
    &=\frac{1}{\lambda_{s,i}}\left|\E\left[\left(\partial_{\zeta}^{2}\sigma\left(\mu_{i}+\lambda_{s,i}\zeta\right)\right)\zeta\right]\right|\\
    &=\frac{1}{\lambda_{s,i}}\left|\E\left[\sigma\left(\mu_{i}+\lambda_{s,i}\zeta\right)\left(\zeta^{3}-3\zeta\right)\right]\right|\\
    &\le \frac{1}{\lambda_{s,i}}\E\left[\left|\zeta^{3}-3\zeta\right|\right]\\
    &\le \frac{2\sqrt{2/\pi}}{\lambda_{s,i}},
\end{align*}
where the second equality is by Lemma \ref{lem:hermite}-(ii) with $d=1$, the third line is by H\"{o}lder's inequality, and the last inequality holds by $\E[|\zeta^{3}-3\zeta|]=(1+4e^{-3/2})\sqrt{2/\pi}< 2\sqrt{2/\pi}$.
By the fact \eqref{eq:proof:lem:localLipschitz:variational}, for some $\bu'''=(\bu_{1}''',\ldots,\bu_{n}''')\in\bbS^{np-1}$,
\begin{align*}
    \sum_{i=1}^{n}\frac{1}{2n}\left\langle \bLambda\bz_{i},\bu_{i}\right\rangle \int_{0}^{t}\E\left[\left(\sigma\left(1-\sigma\right)\right)'\left(\mu_{i}+\lambda_{s,i}\zeta\right)\lambda_{s,i}\zeta\right]\diff s
    &\le 2\sum_{i=1}^{n}\frac{\sqrt{2t/\pi}}{n}\left\langle \frac{1}{\sqrt{\left\langle\bLambda\bz_{i},\bz_{i}\right\rangle}}\bLambda\bz_{i},\bu_{i}'''\right\rangle\\
    &=\frac{2\sqrt{2t/\pi}}{n}\sqrt{\sum_{i=1}^{n}\left\|\frac{1}{\sqrt{\left\langle \bLambda\bz_{i},\bz_{i}\right\rangle}}\bLambda\bz_{i}\right\|^{2}}\\
    &\le \frac{2\sqrt{2t/\pi}}{n}\sqrt{\sum_{i=1}^{n}\left(\frac{1}{\left\langle \bLambda\bz_{i},\bz_{i}\right\rangle}\left\|\bLambda\bz_{i}\right\|^{2}\right)}\\
    &\le 2\sqrt{\frac{2t\|\bSigma\|}{\pi n}}.
\end{align*}

(Step 4) Therefore, for all $\bu\in\bbS^{np-1}$,
\begin{align}
    \left\langle \nabla_{\bz}G_{t}^{\btheta}(\bz),\bu\right\rangle
    \le \sqrt{\frac{2t\left\|\bSigma\right\|}{\pi n}}\left(4+\sqrt{\frac{R^{2}}{16n}\sum_{i=1}^{n}\left\langle \bLambda\bz_{i},\bz_{i}\right\rangle}\right).
\end{align}
This is the desired conclusion.
\end{proof}

The next lemma is a version of Lemmas 3.1 and 3.2 of \citet{adamczak2015note}.

\begin{lemma}\label{lem:tech:adamczak}
    Let $\{S_{v};v>0\}$ be a sequence of real-valued integrable random variables, $Z$ be a real-valued integrable random variable, and $a,b>0$ such that for all $v>0$ and $u>0$,
    \begin{equation*}
        \Pr\left(|S_{v}-\E[S_{v}]|\ge s\right)\le 2\exp\left(-\frac{u^{2}}{a+\sqrt{bv}}\right)
    \end{equation*}
    and 
    \begin{equation*}
        \Pr\left(S_{v}\neq Z\right)\le 2\exp\left(-v/b\right).
    \end{equation*}
    Then for any $v>0$,
    \begin{equation}
        \Pr\left(|Z-\E[Z]|\ge v\right)\le 2\exp\left(-\frac{1}{6}\min\left\{\frac{v^{2}}{(44a)^{2}},\frac{v}{22b}\right\}\right).
    \end{equation}
\end{lemma}

\begin{proof}
    Let $\mathrm{Med}(Z)$ denote the median of $Z$.
    By the proof of Lemma 3.1 of \citet{adamczak2015note}, we know that for any $v>0$,
    \begin{equation}
        \Pr\left(|Z-\mathrm{Med}(Z)|\ge v\right)\le 27\exp\left(-\min\left\{\frac{v^{2}}{4a^{2}},\frac{v}{b}\right\}\right).
    \end{equation}
    We also obtain that
    \begin{equation*}
        \left|\E[Z]-\mathrm{Med}(Z)\right|\le \E\left[\left|Z-\mathrm{Med}(Z)\right|\right]\le 27\int_{0}^{\infty}\exp\left(-\min\left\{\frac{v^{2}}{4a^{2}},\frac{v}{b}\right\}\right)\diff v\le 27(\sqrt{\pi}a+b).
    \end{equation*}
    For $v>(22/21)(27(\sqrt{\pi}a+b))$, we have
    \begin{align*}
        \Pr\left(\left|Z-\E[Z]\right|\ge v\right)\le \Pr\left(\left|Z-\mathrm{Med}(Z)\right|\ge v/22\right)&\le \min\left\{1,27\exp\left(-\min\left\{\frac{v^{2}}{(44a)^{2}},\frac{v}{22b}\right\}\right)\right\}\\
        &\le \min\left\{1,2\exp\left(-\frac{1}{5}\min\left\{\frac{v^{2}}{(44a)^{2}},\frac{v}{22b}\right\}\right)\right\}.
    \end{align*}
    If $v\le (22/21)(27(\sqrt{\pi}a+b))$, then it is sufficient to see that
    \begin{equation}
        \frac{1}{\log2}\min\left\{\frac{v^{2}}{(44a)^{2}},\frac{v}{22b}\right\}\le 6.
    \end{equation}
    If $a\ge b$,
    \begin{equation*}
        \frac{1}{\log2}\min\left\{\frac{v^{2}}{(44a)^{2}},\frac{v}{22b}\right\}\le \frac{1}{\log2}\frac{(22/21)^{2}(27(\sqrt{\pi}a+b))^{2}}{(44a)^{2}}=\frac{(27)^{2}(\sqrt{\pi}+b/a)^{2}}{(42)^{2}\log2}< 5,
    \end{equation*}
    and if $b\ge a$,
    \begin{equation*}
        \frac{1}{\log2}\min\left\{\frac{v^{2}}{(44a)^{2}},\frac{v}{22b}\right\}\le \frac{1}{\log2}\frac{27(\sqrt{\pi}a+b)}{21b}\le \frac{27(\sqrt{\pi}+1)}{21\log2}< 6.
    \end{equation*}
    We obtain the conclusion.
\end{proof}

\begin{proof}[Proof of Proposition \ref{prop:bernstein}]
We only prove the first statement.
The proof is partially adopted from \citet{adamczak2015note}.
We have that for any $\bz,\bz'\in\R^{np}$,
\begin{equation}
    \sup_{\btheta\in\PS}G_{t}^{\btheta}(\bz)-\sup_{\btheta\in\PS}G_{t}^{\btheta}(\bz')\le \sqrt{\frac{2t\left\|\bSigma\right\|}{\pi n}}\left(4+\sqrt{\frac{R^{2}}{16n}\left(\sum_{i=1}^{n}\left\langle \bLambda\bz_{i},\bz_{i}\right\rangle\right)\vee\left(\sum_{i=1}^{n}\left\langle \bLambda\bz_{i}',\bz_{i}'\right\rangle\right)}\right)\left\|\bz-\bz'\right\|.
\end{equation}
We fix a positive number $v>0$, a convex set 
\begin{equation}
    B:=\left\{\bz\in\R^{np};\sqrt{\sum_{i=1}^{n}\left\langle \bLambda\bz_{i},\bz_{i}\right\rangle}<\E\left[\sqrt{\sum_{i=1}^{n}\left\langle \bLambda\bZ_{i},\bZ_{i}\right\rangle}\right]+\sqrt{v\|\bSigma\|}\right\}\subset\R^{np}
\end{equation}
and $((16\sqrt{n}/R+\E[\sqrt{(\sum_{i=1}^{n}\langle \bLambda\bZ_{i},\bZ_{i}\rangle)}])/\sqrt{\|\bSigma\|}+\sqrt{v})$-Lipschitz function $\varphi:\R^{np}\to\R$ such that for all $\bz\in B$,
\begin{equation}
    \varphi(\bz)=\sqrt{\frac{16\pi n^{2}}{2tR^{2}\|\bSigma\|^{2}}}\sup_{\btheta\in\PS}G_{t}^{\btheta}(\bz).
\end{equation}
Then by the concentration property and the Cauchy--Schwarz inequality, for all $u>0$,
\begin{align}
    \Pr\left(\left|\varphi(\bZ)-\E[\varphi(\bZ)]\right|\ge u\right)&\le 2\exp\left(-\frac{u^{2}}{K^{2}\left(\left(16\sqrt{n}/R+\E\left[\sqrt{(\sum_{i=1}^{n}\langle \bLambda\bZ_{i},\bZ_{i}\rangle)}\right]\right)/\sqrt{\|\bSigma\|}+\sqrt{v}\right)^{2}}\right)\notag\\
    &\le 2\exp\left(-\frac{u^{2}}{K^{2}\left(\left(16\sqrt{n}/R+\sqrt{n\bSigma}\right)/\sqrt{\|\bSigma\|}+\sqrt{v}\right)^{2}}\right).
\end{align}
The concentration property and the $1$-Lipschitz continuity of $\sqrt{\sum_{i=1}^{n}\langle \bLambda\bz_{i},\bz_{i}\rangle}/\sqrt{\|\bSigma\|}$ also yield that for all $v>0$,
\begin{equation}
    \Pr\left(\sqrt{\left(\sum_{i=1}^{n}\left\langle \bLambda\bZ_{i},\bZ_{i}\right\rangle\right)}\ge  \E\left[\sqrt{\left(\sum_{i=1}^{n}\left\langle \bLambda\bZ_{i},\bZ_{i}\right\rangle\right)}\right]+\sqrt{v\|\bSigma\|}\right)\le 2\exp\left(-\frac{v}{K^{2}}\right).
\end{equation}
Hence by Lemma \ref{lem:tech:adamczak} with $a=K((16\sqrt{n}/R)+\sqrt{n\tr(\bSigma)})/\sqrt{\|\bSigma\|}$ and $b=K^{2}$, for all $u>0$,
\begin{align*}
    &\Pr\left(\sqrt{\frac{16\pi n^{2}}{2tR^{2}\|\bSigma\|^{2}}}\left|\sup_{\btheta\in\PS}G_{t}^{\btheta}(\bZ)-\E\left[\sup_{\btheta\in\PS}G_{t}^{\btheta}(\bZ)\right]\right|\ge u\right)\\
    &\le 2\exp\left(-\frac{1}{6K^{2}}\min\left\{\frac{u^{2}}{44^{2}\left(16\sqrt{n}/R+\sqrt{n\tr(\bSigma)}\right)^{2}/\|\bSigma\|^{2}},\frac{u}{22}\right\}\right).
\end{align*}
Therefore,
\begin{align*}
    &\Pr\left(\left|\sup_{\btheta\in\PS}G_{t}^{\btheta}(\bZ)-\E\left[\sup_{\btheta\in\PS}G_{t}^{\btheta}(\bZ)\right]\right|\ge u\right)\\
    &\le 2\exp\left(-\frac{1}{6K^{2}}\min\left\{\frac{n^{2}u^{2}}{11^{2}R^{2}(2t/\pi)\left(\frac{16\sqrt{n}}{R}+\sqrt{n\tr(\bSigma)}\right)^{2}},\frac{2nu}{11R\sqrt{2t/\pi}\|\bSigma\|}\right\}\right)\\
    &\le 2\exp\left(-\frac{n}{6K^{2}}\min\left\{\frac{u^{2}}{121(2t/\pi)\left(16+R\sqrt{\tr(\bSigma)}\right)^{2}},\frac{u}{5R\sqrt{t}\|\bSigma\|}\right\}\right)\\
    &\le 2\exp\left(-\frac{n}{6K^{2}}\min\left\{\frac{u^{2}}{155t\left(256+R^{2}\tr\left(\bSigma\right)\right)},\frac{u}{5\sqrt{t}R\|\bSigma\|}\right\}\right).
\end{align*}
Hence we obtain the conclusion.
\end{proof}
\subsection{Proof of Proposition \ref{prop:expsup}}
\begin{lemma}\label{lem:absolConst}
    For arbitrary function $f\in\mathcal{C}_{b}^{2}(\R;\R)$ and $t\in[0,1]$, we have
    \begin{equation}
        \int_{0}^{t}\frac{1}{s}\int_{-\infty}^{\infty}f\left(\sqrt{s}x\right)\left(x^{2}-1\right)\frac{\exp(-x^{2}/2)}{\sqrt{2\pi}}\diff x\diff s=2\left(\int_{-\infty}^{\infty}f(\sqrt{t}x)\frac{\exp(-x^{2}/2)}{\sqrt{2\pi}}\diff x-f(0)\right).
    \end{equation}
\end{lemma}

\begin{proof}
    We change variables $s=e^{-2r}$ ($\diff s=-2e^{-2r}\diff r=-2s\diff r$) and then for $T_{0}:=\log(1/\sqrt{t})$,
    \begin{equation}
        \int_{0}^{t}\frac{1}{s}\int_{-\infty}^{\infty}f\left(\sqrt{s}x\right)\left(x^{2}-1\right)\frac{\exp\left(-x^{2}/2\right)}{\sqrt{2\pi}}\diff x\diff s=2\int_{T_{0}}^{\infty}\int_{-\infty}^{\infty}f(e^{-r}x)\left(x^{2}-1\right)\frac{\exp\left(-x^{2}/2\right)}{\sqrt{2\pi}}\diff x\diff r.\label{eq:lemma:absolConst:1}
    \end{equation}
    Let us consider a $1$-dimensional stationary Ornstein--Uhlenbeck process $\{X_{t};t\ge0\}$ defined by the following stochastic differential equation:
    \begin{equation}
        \diff X_{t}=-X_{t}\diff t+\sqrt{2}\diff W_{t},\ X_{0}\sim\Gaussian(0,1),
    \end{equation}
    where $\{W_{t};t\ge0\}$ is a $1$-dimensional standard Brownian motion independent of $X_{0}$.  
    Note that for each $t\ge0$, $X_{t}\sim \Gaussian(0,1)$.

    For any $T_{1}\ge T_{0}$,
    \begin{align*}
        \E\left[f\left(e^{-T_{1}}X_{T_{1}}\right)\right]&=\E\left[f\left(e^{-T_{0}}X_{T_{0}}\right)\right]+\int_{T_{0}}^{T_{1}}\E\left[\left.-X_{s}\partial_{x}f(e^{-s}x)+\partial_{s}f(e^{-s}x)+\partial_{x}^{2}f(e^{-s}x)\right|_{x=X_{s}}\right]\diff s\\
        &=\E\left[f\left(e^{-T_{0}}X_{T_{0}}\right)\right]+\int_{T_{0}}^{T_{1}}\E\left[\left.-2X_{s}\partial_{x}f(e^{-s}x)\right|_{x=X_{s}}+f(e^{-s}X_{s})(X_{s}^{2}-1)\right]\diff s\\
        &=\E\left[f\left(e^{-T_{0}}X_{T_{0}}\right)\right]-\int_{T_{0}}^{T_{1}}\E\left[f(e^{-s}X_{s})(X_{s}^{2}-1)\right]\diff s,
    \end{align*}
    by It\^{o}'s formula, Lemma \ref{lem:hermite}-(ii) with $d=0$, and the fact that $\partial_{s}f(e^{-s}x)=-e^{-s}xf'(e^{-s}x)=-x\partial_{x}f(e^{-s}x)$.
    Taking the limit $T_{1}\to\infty$, we obtain
    \begin{equation}
        \int_{T_{0}}^{\infty}\E\left[f(e^{-s}X_{s})(X_{s}^{2}-1)\right]\diff s=\E\left[f\left(e^{-T_{0}}X_{T_{0}}\right)\right]-f(0).
    \end{equation}
    This left-hand side is identical to the right-hand side of Eq.~\eqref{eq:lemma:absolConst:1} up to constant $2$.
\end{proof}

We show a concise version of Lemma 5 of \citet{meir2003generalization} for self-containedness; see also Theorem 4.12 of \citet{ledoux1991probability} and Proposition 4.3 of \citet{bach2024learning}.
\begin{lemma}\label{lem:rademacher:contraction}
    For arbitrary sequences of functions $\{a_{i}\}$ with $a_{i}:\PS\to\R$, a function $b:\PS\to\R$, and a sequence of $L_{i}$-Lipschitz functions $\{\varphi_{i}\}$ with $\varphi_{i}:\R\to\R$,
    \begin{equation}
        \E\left[\sup_{\btheta\in\PS}\left(b\left(\btheta\right)+\sum_{i=1}^{n}\varepsilon_{i}\varphi_{i}\left(a_{i}\left(\btheta\right)\right)\right)\right]\le \E\left[\sup_{\btheta\in\PS}\left(b\left(\btheta\right)+\sum_{i=1}^{n}\varepsilon_{i}L_{i}a_{i}\left(\btheta\right)\right)\right],
    \end{equation}
    where $\{\varepsilon_{i};i=1,\ldots,n\}$ is an i.i.d.~sequence of Rademacher random variables.
\end{lemma}

\begin{proof}
    We prove the statement by induction as \citet{meir2003generalization}.
    We consider $n=0$, and then the statement immediately holds.
    Let us assume that the statement is true for some $n=m-1$ with $m\ge 2$.
    It holds that
    \begin{align*}
        &\E\left[\sup_{\btheta\in\PS}\left(b\left(\btheta\right)+\sum_{i=1}^{m}\varepsilon_{i}\varphi_{i}\left(a_{i}\left(\btheta\right)\right)\right)\right]\\
        &=\frac{1}{2}\E\left[\sup_{\btheta\in\PS}\left(b\left(\btheta\right)+\sum_{i=1}^{m-1}\varepsilon_{i}\varphi_{i}\left(a_{i}\left(\btheta\right)\right)+\varphi_{m}\left(a_{m}\left(\btheta\right)\right)\right)\right]\\
        &\quad+\frac{1}{2}\E\left[\sup_{\btheta\in\PS}\left(b\left(\btheta\right)+\sum_{i=1}^{m-1}\varepsilon_{i}\varphi_{i}\left(a_{i}\left(\btheta\right)\right)-\varphi_{m}\left(a_{m}\left(\btheta\right)\right)\right)\right]\\
        &=\E\left[\sup_{\btheta,\btheta'\in\PS}\left(\frac{b\left(\btheta\right)+b\left(\btheta'\right)}{2}+\sum_{i=1}^{m-1}\varepsilon_{i}\frac{\varphi_{i}\left(a_{i}\left(\btheta\right)\right)+\varphi_{i}\left(a_{i}\left(\btheta'\right)\right)}{2}+\frac{\varphi_{m}\left(a_{m}\left(\btheta\right)\right)-\varphi_{m}\left(a_{m}\left(\btheta'\right)\right)}{2}\right)\right]\\
        &\le\E\left[\sup_{\btheta,\btheta'\in\PS}\left(\frac{b\left(\btheta\right)+b\left(\btheta'\right)}{2}+\sum_{i=1}^{m-1}\varepsilon_{i}\frac{\varphi_{i}\left(a_{i}\left(\btheta\right)\right)+\varphi_{i}\left(a_{i}\left(\btheta'\right)\right)}{2}+\frac{L_{m}\left|a_{m}\left(\btheta\right)-a_{m}\left(\btheta'\right)\right|}{2}\right)\right]\\
        &=\E\left[\sup_{\btheta,\btheta'\in\PS}\left(\frac{b\left(\btheta\right)+b\left(\btheta'\right)}{2}+\sum_{i=1}^{m-1}\varepsilon_{i}\frac{\varphi_{i}\left(a_{i}\left(\btheta\right)\right)+\varphi_{i}\left(a_{i}\left(\btheta'\right)\right)}{2}+\frac{L_{m}\left(a_{m}\left(\btheta\right)-a_{m}\left(\btheta'\right)\right)}{2}\right)\right]\\
        &=\frac{1}{2}\E\left[\sup_{\btheta\in\PS}\left(b\left(\btheta\right)+\sum_{i=1}^{m-1}\varepsilon_{i}\varphi_{i}\left(a\left(\btheta\right)\right)+L_{m}a_{m}\left(\btheta\right)\right)\right]\\
        &\quad+\frac{1}{2}\E\left[\sup_{\btheta\in\PS}\left(b\left(\btheta\right)+\sum_{i=1}^{m-1}\varepsilon_{i}\varphi_{i}\left(a\left(\btheta\right)\right)-L_{m}a_{m}\left(\btheta\right)\right)\right]\\
        &\le \frac{1}{2}\E\left[\sup_{\btheta\in\PS}\left(b\left(\btheta\right)+\sum_{i=1}^{m-1}\varepsilon_{i}L_{i}a_{i}\left(\btheta\right)+L_{m}a_{m}\left(\btheta\right)\right)\right]\\
        &\quad+\frac{1}{2}\E\left[\sup_{\btheta\in\PS}\left(b\left(\btheta\right)+\sum_{i=1}^{m-1}\varepsilon_{i}L_{i}a_{i}\left(\btheta\right)-L_{m}a_{m}\left(\btheta\right)\right)\right]\\
        &=\E\left[\sup_{\btheta\in\PS}\left(b\left(\btheta\right)+\sum_{i=1}^{m}\varepsilon_{i}L_{i}a_{i}\left(\btheta\right)\right)\right],
    \end{align*}
    where the second last inequality is given by the induction hypothesis.
\end{proof}

\begin{proof}[Proof of Proposition \ref{prop:expsup}]
    We only give proof of the first statement as the second statement is parallel.
    Proposition 4.2 of \citet{bach2024learning} yields that for each $t\in(0,1]$,
    \begin{equation}
        \E_{\bZ}\left[\sup_{\btheta\in\PS}G_{t}^{\btheta}\left(\bZ\right)\right]\le 2\E_{\bvarepsilon,\bZ}\left[\sup_{\btheta\in\PS}\frac{1}{2n}\sum_{i=1}^{n}\varepsilon_{i}\int_{0}^{t}\E_{\BM}\left[\sigma\left(1-\sigma\right)\left(\left\langle \tilde{\BM}_{s}^{\btheta},\bZ_{i}\right\rangle\right)\right]\diff s\left\langle \bLambda\bZ_{i},\bZ_{i}\right\rangle\right],
    \end{equation}
    where $\bvarepsilon=(\varepsilon_{1},\ldots,\varepsilon_{n})$ is an $\R^{n}$-valued random vector whose components are independent Rademacher random variables, which is also independent of $\BM$ and $\bZ$.
    We set $\lambda_{i}=\lambda_{i}(\bz_{i})=\sqrt{\langle \bLambda\bz_{i},\bz_{i}\rangle},\nu_{i}=\nu_{i}(\bz_{i},\btheta)=\lambda_{i}^{-1}\langle \bfU\bLambda^{1/2}\bz_{i},\btheta\rangle$, and $\zeta\sim\Gaussian(0,1)$; then $\E_{\BM}[\sigma(1-\sigma)(\langle \tilde{\BM}_{s}^{\btheta},\bZ_{i}\rangle)]=\E_{\zeta}[\sigma(1-\sigma)(\lambda_{i}(\nu_{i}+\sqrt{s}\zeta))].$
    By Lemma \ref{lem:hermite}, we have that for any $s>0$ and $i=1,\ldots,n$, 
    \begin{align*}
        \partial_{\nu_{i}}\E_{\zeta}\left[\sigma\left(1-\sigma\right)\left(\lambda_{i}(\nu_{i}+\sqrt{s}\zeta)\right)\right]
        &=\E_{\zeta}\left[\left(\sigma\left(1-\sigma\right)\right)'\left(\lambda_{i}(\nu_{i}+\sqrt{s}\zeta)\right)\right]\lambda_{i} \\
        &=\E_{\zeta}\left[(\partial_{\zeta}^{2}\sigma)\left(\lambda_{i}(\nu_{i}+\sqrt{s}\zeta)\right)\right]\frac{\lambda_{i}}{s\lambda_{i}^{2}}\\
        &=\E_{\zeta}\left[\sigma\left(\lambda_{i}(\nu_{i}+\sqrt{s}\zeta)\right)\left(\zeta^{2}-1\right)\right]\frac{1}{s\lambda_{i}};
    \end{align*}
    the interchange of limiting operations in the first line holds as the partial derivative of the integrand on the left-hand side is dominated by $\lambda_{i}\|(\sigma(1-\sigma))'\|_{\infty}<\infty$ on $\R$.
    The left-hand side itself is also dominated by $\lambda_{i}\|(\sigma(1-\sigma))'\|_{\infty}$ on $(0,t]$.
    Therefore, Lemma \ref{lem:absolConst} gives that for any $t\in(0,1]$,
    \begin{align*}
        \left|\partial_{\nu_{i}}\int_{0}^{t}\E_{\zeta}\left[\sigma\left(1-\sigma\right)\left(\lambda_{i}(\nu_{i}+\sqrt{s}\zeta)\right)\right]\diff s\right|
        &= \left|\int_{0}^{t}\partial_{\nu_{i}}\E_{\zeta}\left[\sigma\left(1-\sigma\right)\left(\lambda_{i}(\nu_{i}+\sqrt{s}\zeta)\right)\right]\diff s\right|\\
        &=\left|\int_{0}^{t} \E_{\zeta}\left[\sigma\left(\lambda_{i}(\nu_{i}+\sqrt{s}\zeta)\right)\left(\zeta^{2}-1\right)\right]\frac{1}{s\lambda_{i}}\diff s\right|\\
        &= \frac{2}{\lambda_{i}}\left|\E\left[\sigma(\lambda_{i}(\nu_{i}+\sqrt{t}\zeta))\right]-\sigma(\lambda_{i}(\nu_{i}))\right|\\
        &\le \frac{2}{\lambda_{i}}.
    \end{align*}
    Therefore, for any $t\in(0,1]$,
    \begin{equation}
        \left|\partial_{\nu_{i}}\int_{0}^{t}\E_{\zeta}\left[\sigma\left(1-\sigma\right)\left(\lambda_{i}(\nu_{i}+\sqrt{s}\zeta)\right)\right]\diff s\left\langle\bLambda\bz_{i},\bz_{i}\right\rangle\right|\le 2\lambda_{i}.
    \end{equation}
    By Lemma \ref{lem:rademacher:contraction} (with $a_{i}=\nu_{i}$, $b=0$, and $L_{i}=2\lambda_{i}$), we obtain that for all $t\in(0,1]$ and $\bz\in\R^{np}$,
    \begin{align*}
        \E_{\bvarepsilon}\left[\sup_{\btheta\in\PS}\sum_{i=1}^{n}\varepsilon_{i}\int_{0}^{t}\E_{\BM}\left[\sigma\left(1-\sigma\right)\left(\left\langle \tilde{\BM}_{s}^{\btheta},\bz_{i}\right\rangle\right)\right]\diff s\left\langle \bLambda\bz_{i},\bz_{i}\right\rangle\right]
        &\le 2\E_{\bvarepsilon}\left[\sup_{\btheta\in\PS}\sum_{i=1}^{n}\varepsilon_{i}\left\langle \bfU\bLambda^{1/2}\bz_{i},\btheta\right\rangle\right]\\
        &= 2R\E_{\bvarepsilon}\left[\left\| \sum_{i=1}^{n}\varepsilon_{i}\bLambda^{1/2}\bz_{i}\right\|\right]\\
        &\le 2R\E_{\bvarepsilon}\left[\left\| \sum_{i=1}^{n}\varepsilon_{i}\bLambda^{1/2}\bz_{i}\right\|^{2}\right]^{1/2}\\
        &=2R\left(\sum_{i=1}^{n}\left\langle\bLambda\bz_{i},\bz_{i}\right\rangle\right)^{1/2}
    \end{align*}
    and thus
    \begin{equation}
         \E_{\bZ}\left[\sup_{\btheta\in\PS}G_{t}^{\btheta}\left(\bZ\right)\right]\le \frac{2R}{n}\sqrt{\sum_{i=1}^{n}\E_{\bZ}\left[\left\langle \bLambda\bZ_{i},\bZ_{i}\right\rangle\right]}
         \le 2R\sqrt{\frac{\tr(\bSigma)}{n}}
    \end{equation}
    by Jensen's inequality.
    Therefore, the statement holds true.
\end{proof}
\end{appendices}

\bibliographystyle{apalike}
\bibliography{bibliography}

\end{document}